\documentclass{article}

\usepackage{amsfonts}
\usepackage{amssymb}
\usepackage[english]{babel}
\usepackage{algorithm}
\usepackage{algorithmic}
\usepackage{booktabs}
\usepackage{nicefrac}
\usepackage{authblk}
\usepackage[small]{titlesec}
\usepackage{tikz,pgfplots}
\usepackage{subcaption}

\usepackage[
,breaklinks=true  
,colorlinks       
,linkcolor=blue
,urlcolor=red!50!black    
,citecolor=black!40!green
,bookmarks         
,bookmarksnumbered,
pdftex=true,              
plainpages=false,      
hypertexnames=false,     
pdfpagelabels=true,   
hyperindex=true,
pdfa=true
]{hyperref}
\usepackage[leqno]{amsmath}
\usepackage{amsthm}

\newtheorem{theorem}{Theorem}[section]

\newtheorem{lemma}[theorem]{Lemma}

\newtheorem{corollary}[theorem]{Corollary}

\newtheoremstyle{mystyle}
{}
{}
{\normalfont}
{}
{\normalfont\itshape}
{.}
{ }
{}

\theoremstyle{mystyle}
\newtheorem{example}[theorem]{Example}

{\begin{itemize}%
		\setlength{\itemsep}{1pt}%
		\setlength{\parskip}{0pt}}%
{\end{itemize}}

\newcommand{\cZ}{\mathcal{Z}}

\newcommand{\R}{\mathbb{R}}
\newcommand{\Z}{\mathbb{Z}}
\newcommand{\N}{\mathbb{N}}

\newcommand{\SP}{\textup{\tiny SP}}

\renewcommand{\d}{\text{d}}
\newcommand{\conv}{\operatorname{conv}}

\newcommand{\ie}{i.e., }
\newcommand{\eg}{e.g., }

\newcommand{\embed}{\hookrightarrow}

\renewcommand{\AA}{\mathcal{A}}
\newcommand{\II}{\mathcal{I}}

\providecommand{\keywords}[1]{\textbf{Keywords.} #1}

\title{\bf Parabolic optimal control
	problems with combinatorial switching constraints \\ Part III: Branch-and-bound algorithm\thanks{This work has partially been supported by Deutsche Forschungsgemeinschaft (DFG) under grant nos.~BU~2313/7-1 and ME~3281/10-1.}}
\author{Christoph~Buchheim, Alexandra~Gr\"utering, and Christian~Meyer\footnote{\{christoph.buchheim,alexandra.gruetering,christian.meyer\}@math.tu-dortmund.de}}
\date{\vspace{-5ex}}
\affil{Department of Mathematics, TU Dortmund University, Germany}

\tikzset{
	treenode/.style = {align=center, inner sep=0pt, text centered,
		font=\sffamily},
	arn/.style = {treenode, circle, black, font=\sffamily, draw=black,
		fill=white, text width=3ex},
	arnbold/.style = {treenode, circle, line width=0.3mm, font=\sffamily\bfseries, draw=black,
		fill=white, text width=3ex},
}

\begin{document}
\maketitle
\begin{abstract}
  We present a branch-and-bound algorithm for globally solving
  parabolic optimal control problems with binary switches that have
  bounded variation and possibly need to satisfy further combinatorial
  constraints. More precisely, for a given
  tolerance~\mbox{$\varepsilon>0$}, we show how to compute in finite
  time an $\varepsilon$-optimal solution in function space,
  independently of any prior discretization. The main
  ingredients in our approach are an appropriate branching strategy in
  infinite dimension, an a posteriori error estimation in order to
  obtain safe dual bounds, and an adaptive refinement strategy in
  order to allow arbitrary switching points in the limit. The
  performance of our approach is demonstrated by extensive
  experimental results.

  \medskip
  
\end{abstract}

\section{Introduction}    

Optimal control problems with discrete switches have recently become
an increasing focus of research. Most approaches presented in the
literature, however, produce only heuristic solutions without
any quality guarantee. The
well-known Sum-Up Rounding approach~\cite{SA12,KLM20} computes binary
switching patterns by first solving a convex relaxation of the problem
and then approximating the resulting continuous switching by a binary
one. This approach often requires a large number of
switchings when trying to come close to the optimal continuous
solution. In particular, it cannot deal with an explicit
bound on the number of switchings, let alone with more complex
combinatorial constraints. If such constraints need to be satisfied,
the Combinatorial Integral Approximation approach~\cite{SA05} can be
applied. Since the latter again tries to approximate a given continuous
control by a feasible binary one, this approach does not lead to optimal
solutions to the original problem in general, even when a best-possible
approximation can be computed. Other approaches aim at optimizing the switching times of the discrete switches 
\cite{GER06,EWA06,JM11,FMO13,HR16,ROL17,SOG17},
or use non-smooth penalty techniques, partly in combination with 
convexification, to impose the switching structure, see, \eg \cite{CK14,CK16,CTW18,CKK18,Wac19} and the references therein. However, both strategies in general lead to non-convex problems with potentially multiple local minima and a convexifcation of the arising problems may destroy the switching structure of the optimal solution.

In this paper, we present a branch-and-bound approach for solving
parabolic optimal control problems with combinatorial switching
constraints to \emph{global} optimality. More precisely, we consider problems
of the form
\begin{equation}\tag{P}\label{eq:optprob}
\left\{\quad
\begin{aligned}
\text{min} \quad & J(y,u) = \tfrac{1}{2}\, \|y - y_{\textup{d}}\|_{L^2(Q)}^2 + \tfrac{\alpha}{2}\,\|u-\tfrac 12\|^2_{L^2(0,T)}\\
\text{s.t.} \quad & 
\begin{aligned}[t]
\partial_t y(t,x) - \Delta y(t,x) &= u(t) \,\psi(x) & & \text{in } Q := \Omega \times (0,T),\\
y(t,x) &= 0 & & \text{on } \Gamma := \partial\Omega \times (0,T),\\  
y(0,x) &= y_0(x) & & \text{in } \Omega,
\end{aligned}\\
\text{and} \quad & u \in D\;.
\end{aligned}
\quad \right.
\end{equation}
Here
$T > 0$ is a given final time and $\Omega\subset \R^d$, $d\in \N$, is
a bounded domain, \ie a bounded, open, and connected set, with
Lipschitz boundary~$\partial \Omega$ in the sense of
\cite[Def.~1.2.2.1]{GRIS85}. The form function $\psi\in
H^{-1}(\Omega)$ and the initial state~$y_0 \in L^2(\Omega)$
are given. Moreover,~$y_{\textup{d}} \in L^2(Q)$ is a given desired
state and $\alpha \geq 0$ is a Tikhonov parameter weighting the
deviation from~$\tfrac 12$. Finally,
\[
D \subseteq BV(0,T;\{0,1\}):=\big\{ u \in BV(0,T)\colon u(t) \in \{0,1\} \text{ f.a.a.\ } t \in (0,T)\big\}
\]
denotes the set of feasible \emph{switching controls} and is supposed to satisfy the following assumptions: 
\begin{align}
& \text{$D$ is a bounded set in $BV(0,T)$,} \tag{D1}\label{eq:D1} \\
& \text{$D$ is closed in $L^p(0,T)$ for some fixed $p \in [1,\infty)$.} \tag{D2}\label{eq:D2}
\end{align}
Here $BV(0,T)$ denotes the set of all functions in~$L^1(0,T)$ with bounded variation,
equipped with the norm
$\|u\|_{BV(0,T)} := \|u\|_{L^1(0,T)} + |u|_{BV(0,T)}$; see \eg \cite{ATT14} for details on the space of functions with bounded variation.
For simplicity, we restrict ourselves to the case of one binary
switch in~\eqref{eq:optprob}, but our main results are easily extended
to the case of multiple switches, i.e., to~$D\subset
BV(0,T;\{0,1\}^k)$ for some~$k\in\N$.

Assumption~\eqref{eq:D1} is crucial in our context. Without this
condition, it would be possible to approximate any control~$u$
with~$u(t)\in[0,1]\text{ f.a.a.\ } t \in (0,T)$ arbitrarily well by a
binary switch~$u$ using an increasing number of switchings, e.g., by
applying the Sum-Up Rounding approach mentioned above. Moreover, in our
branch-and-bound algorithm, the bounded variation is essential in
order to ensure the effectiveness of fixings.

In \cite{partI,partII}, we present tailored convexifications of
\eqref{eq:optprob} and an outer approximation approach to solve the
resulting relaxations. The core of the approach is the generation of
linear cutting planes describing the closed convex hull of~$D$
in~$L^p(0,T)$. While the overall approach is very general, the
specific shape of the cutting planes is problem-dependent. In
particular, we devise results for the case of bounded total variation
(without further constraints) and for the case where the switching points of the control~$u$ must satisfy given linear constraints. The latter case comprises the so-called mininum dwell-time contraints~\cite{ZRS20}.

Building on the results of~\cite{partI,partII}, our aim is thus to
determine \emph{globally} optimal solutions for problems of
type~\eqref{eq:optprob} that are independent of any prior discretization, using a branch-and-bound approach. For this,
we start from the convex relaxations of~\eqref{eq:optprob} studied
in~\cite{partI}. These convex relaxations correspond to the
root nodes in our branch-and-bound algorithm. In order to extend this
to a full branch-and-bound algorithm for computing globally optimal
solutions (at least in the limit), we have to overcome several
obstacles:
\begin{itemize}
\item[--] Since we optimize in function space, fixing the value of the
  switch in finitely many points (as is common in finite-dimensional
  branch-and-bound algorithms) has no direct effect, or is not even
  well-defined. We thus have to take the bounded variation into
  account in order to obtain implicit restrictions on the set of admissible
  controls in the nodes of the branch-and-bound tree; see
  Section~\ref{sec:branching}.
\item[--] The fixing of the switch at certain points in time leads to
  a non-closed set of admissible controls in the nodes. Moreover, the
  closed convex hulls of these sets are structurally different from
  the admissible controls arising in the root node. We study the most important
  classes of these sets in Section~\ref{sec:convhull}.
\item[--] In~\cite{partII}, we devised a semi-smooth Newton
  method to solve the root relaxation. However, with an increasing
  number of fixings, this method became less stable, so that we now
  propose to solve all subproblems by the alternating direction method
  of multipliers; see Section~\ref{sec:comp}.
\item[--] In order to obtain globally optimal solutions, all dual
  bounds computed in the nodes of the branch-and-bound tree must be
  safe. In particular, they need to take discretization errors into
  account. In case the time-mesh independent dual bound is too weak
  to cut off a node, we may either have to branch or to refine the
  temporal grid, depending on the relation between the current primal
  bound and the time-mesh dependent dual bound. The sophisticated
  interplay between branching, error analysis, and adaptive refinement
  is at the core of our proposed approach, it is discussed in
  Section~\ref{sec:numeric}.
\end{itemize}
The main contribution of this paper is to present solutions for the challenges listed above. An extensive experimental evaluation presented in Section~\ref{sec:exp} shows that an effective and stable implementation of the resulting branch-and-bound approach is possible.

\section{Preliminaries}\label{sec:pre}

We first collect some definitions and observations that are needed in the following sections, concerning the solution mapping for the PDE in~\eqref{eq:optprob} as well as functions of bounded variation and special switching constraints.

\subsection{Solution mapping}\label{sec:solmap}

The assumptions on the optimal control problem~\eqref{eq:optprob}
listed above guarantee that, for every control~$u\in D\subset
L^2(0,T)$, the PDE in~\eqref{eq:optprob} admits a unique weak solution
$$y\in W(0,T)
:= H^1(0,T;H^{-1}(\Omega)) \cap L^2(0,T; H^1_0(\Omega))\;;$$
see~\cite[Chapter~3]{Troe10}. To specify the associated solution
operator
\[
S\colon L^2(0,T)\ni u \mapsto y \in W(0,T)\;,
\]
we introduce
the linear and continuous (and thus Fr\'echet differentiable) operator
\[
\Psi\colon L^2(0,T) \to L^2(0,T;H^{-1}(\Omega)), \quad 
(\Psi u) (t,x) = u(t) \psi(x)
\] 
as well as the solution operator $\Sigma : L^2(0,T;H^{-1}(\Omega)) \to
W(0,T)$ of the heat equation with homogeneous initial condition, \ie
given $w\in L^2(0,T;H^{-1}(\Omega))$, $y = \Sigma(w)$ solves
\[
\partial_t y - \Delta y = w \quad \text{in } L^2(0,T;H^{-1}(\Omega)), 
\quad y(0) = 0 \quad\text{in } L^2(\Omega).
\]
Moreover, we introduce the function $\zeta \in W(0,T)$ as solution for
\[
\partial_t \zeta - \Delta \zeta = 0 \quad \text{in } L^2(0,T;H^{-1}(\Omega)), 
\quad \zeta(0) = y_0 \quad\text{in } L^2(\Omega).
\]
Then the solution operator~$S$ is given
by $ S = \Sigma \circ \Psi + \zeta$. In particular, it is affine and
continuous. Using this solution operator, the
problem~\eqref{eq:optprob} can be written as
\begin{equation} \tag{P\text{$'$}} \label{eq:P'}
\left\{\quad
\begin{aligned}
\mbox{min }~ &J(Su,u) \\ 
\mbox{s.t. }~
& u\in D\;.
\end{aligned}\, \right.
\end{equation}

\subsection{Functions of bounded variation}\label{sec:bv}

Functions of bounded variation are of central importance in the
following. In order to deal with such functions, first recall that
each function $u \in BV(0,T)$ admits a right-continuous representative
given by $\hat u(t) = c+\mu((0,t])$, $t\in (0,T)$, where $\mu$ is the
regular Borel measure on~$[0,T]$ associated with the distributional
derivative of $u$ and $c\in \R$ is a constant. Note that~$\hat u$ is
unique on~$(0,T)$.  Here and in the following, with a slight abuse
of notation, we denote this function by the same symbol as the
equivalence class in $BV(0,T)$ and, when it comes to pointwise
evaluations, we always refer to this representative function.
In particular, we will often write constraints in the form~$u(t)=b$
for~$b\in\R$. For~$t\in(0,T)$, this is well-defined by the above
reasoning, it then means~$\hat u(t)=b$, while for~$t=0$, we use the
same notation as shorthand for~$\lim_{t\searrow 0}\hat u(t)=b$.

In this paper, we will mostly deal with binary controls~$u\in
BV(0,T;\{0,1\})$. In this case, the representative~$\hat u$ can be
parameterized through its switching points~$0\leq
t_1\leq \cdots \leq t_\sigma$, where~$\sigma\le\|u\|_{BV(0,T)}$. More formally, if
one already counts $\hat u(0)=\lim_{t\searrow 0} \hat{u}(t)=1$ as one switching
from $0$~to $1$, then the representative can be written in the form
$$u_{t_1,\dots,t_\sigma}(t):=
\begin{cases}
0, & \text{if \,$|\{i \in \{1, \ldots, \sigma\} \colon \, t_i \le t \}|$ is even},\\
1 , & \text{if\, $|\{i \in \{1, \ldots, \sigma\} \colon \, t_i \le t \}|$ is odd};
\end{cases}                
$$ see~\cite{partI} for more details. In the following, we will always regard $u_{t_1,\dots,t_\sigma}(t)$ as a function in~$BV(0,T)$.

\subsection{Examples of switching constraints}\label{sec:switchingcon}

In \cite{partI,partII}, we present tailored convexifications of
\eqref{eq:P'} and an outer approximation algorithm to solve the
resulting relaxations. In particular, we elaborate the details of this
approach for the following two relevant classes of constraints~$D$. By
Assumption~\eqref{eq:D1}, the total number of switchings is bounded,
and the first type of constraint arises when this is the only
restriction. More specifically, we restrict the total variation of the
single switch from above by $\sigma>0$, so that the set of feasible
controls is
\begin{equation}\label{eq:dmax}
  D(\sigma):=\{u\in BV(0,T):u(t)\in \{0,1\} \text{ f.a.a. }t\in
  (0,T),\ |u|_{BV(0,T)}\leq \sigma \}.
\end{equation}
The second type of constraint imposes affine linear relations between
the positions of the switching points of~$u$.
More precisely, for a given polytope~$P\subseteq \R_+^\sigma$, we define
\begin{equation}\label{eq:DP}
  D(P) := \{ u_{t_1,\dots,t_\sigma} \colon (t_1, \ldots, t_\sigma) \in P,~
  0\leq t_1\leq \cdots\leq t_\sigma < \infty\}\;.
\end{equation}
An important special case are the so-called minimum-dwell time
constraints, defined as
\begin{equation}\label{eq:Ds} 
  D(s) := \big\{ u_{t_1,\dots,t_\sigma}\colon t_{i}-t_{i-1}\ge s ~ \forall \,i = 2, \dots,\sigma,~ t_1,\ldots,t_\sigma\geq 0\big\}
\end{equation}
for some given~$s>0$. In words, a minimum time span~$s$ is required between
two consecutive switchings of~$u$.

All sets~$D(\sigma)$ and $D(P)$ defined here, and hence also~$D(s)$, satisfy the general assumptions~\eqref{eq:D1} and~\eqref{eq:D2}; see~\cite[Lemma~3.7 and Lemma~3.10]{partI}. We will investigate these sets under fixings in Section~\ref{sec:convhull} and use~$D(\sigma)$ for our experiments presented in Section~\ref{sec:exp}.

\section{Branch-and-bound algorithm}\label{sec:branching}

In finite-dimensional optimization,
branch-and-bound is the standard approach for solving non-convex
optimization problems to global optimality. First, a dual bound is
computed for the original problem, corresponding to the \emph{root
node} of the branch-and-bound tree. Often, this is done by solving a
convex relaxation of the problem. In case the optimal solution for the
latter is infeasible for the original problem, a \emph{branching} is
applied. In the most abstract form, this means that the set of
feasible solutions is subdivided into two (or more) subsets,
corresponding to the child nodes of the root node. Recursive
application of the branching leads to the so-called
\emph{branch-and-bound tree}. The \emph{bounding} is now applied in
order to reduce the number of nodes in this tree, which leads to a
finite algorithm in many cases: one first needs to obtain so-called
\emph{primal solutions}, i.e., feasible solutions of the original
problem. Each such solution yields (in case of a minimization problem)
a global upper bound on the optimal value of the original
problem. Now, if the dual bound obtained in some
branch-and-bound node is larger than the best known upper bound, it
follows that this node cannot contain any optimal solution, so that
it can be \emph{pruned}, i.e., the entire subtree
rooted at this node can be ignored in the enumeration.

The most natural branching strategy for finite-dimensional binary
optimization problems consists of picking a binary variable having a
fractional value in the optimal solution for the convex relaxation used
for computing the dual bound, and then fixing this variable to zero in
the first child node and to one in the other. However, in the
infinite-dimensional setting considered here, the situation is more
complicated: we need to deal with infinitely many binary variables,
suggesting that an infinite number of function values has to be fixed
in order to uniquely determine a solution for~\eqref{eq:optprob}. In
fact, fixing a pointwise value of~$u$ has no direct effect (or is not
even well-defined) in the function space~$L^p(0,T)$. At this point, we
can exploit Assumption~\eqref{eq:D1}, which yields a finite bound on
the total number of switching points. The relevant restrictions in a
given node of the branch-and-bound tree are now a joint consequence of
the finitely many fixing decisions taken so far and of the
constraint~$u\in D$.

The main challenge is now to describe these resulting
restrictions. Assume that our branching strategy always picks
appropriate time points~$\tau\in(0,T)$ and fixes~$u(\tau)=0$ in the
first subproblem and~$u(\tau)=1$ in the second. Then
all our subproblems, corresponding to the nodes in the
branch-and-bound tree, are problems in $BV(0,T)$ of the form
\begin{equation} \tag{SP} \label{eq:SP}
\left\{\quad
\begin{aligned}
\mbox{inf }~ &J(Su,u)\\ 
\mbox{s.t.}~
& u\in D \\
& u(\tau_j) = c_j \quad \forall j=1,\ldots,N
\end{aligned}\, \right.
\end{equation}
with $(\tau_j,c_j) \in [0,T)\times\{0,1\}$ for $1\leq j \leq N$; see Section~\ref{sec:bv} for the precise meaning of the fixing constraints. In the
following, we denote the feasible set of~\eqref{eq:SP} by
$$D_\SP:=\{u\in D: u(\tau_j) = c_j \ \forall j=1,\ldots,N\}\;,$$
where we
always assume $\tau_1<\cdots<\tau_N$. Note that the set~$D_\SP$ is not
closed in general, and hence the subproblem~\eqref{eq:SP} does not
necessarily admit a global minimizer. However, this is no problem
since we are only interested in the optimal value of~\eqref{eq:SP} in
our branch-and-bound framework. In fact, our approach will produce a series of dual bounds
by convexifying~\eqref{eq:SP} and these covexifications will provide the same (primal) optimal value
of~\eqref{eq:SP} in the limit; see Theorem~\ref{thm:convobj} below. We we consider the convexification
\begin{equation} \tag{SPC} \label{eq:SPC}
\left\{\quad
\begin{aligned}
\mbox{inf }~ &J(Su,u)\\ 
\mbox{s.t.}~
& u\in \overline\conv(D_\SP)
\end{aligned}\, \right.
\end{equation}
of the suproblem~\eqref{eq:SP} in $L^p(0,T)$. Here and in the following, $\overline\conv$ always denotes the closed convex hull in~$L^p(0,T)$.

In a reasonable branching strategy, one may expect that an increasing
number of fixing decisions, taken along a path in the branch-and-bound
tree starting at the root node, leads to a unique solution in the
limit. In particular, the dual bounds obtained in the nodes and the
optimal values subject to the corresponding fixings should converge to
each other. The next result shows that this is guaranteed in our
infinite-dimensional setting if the fixing positions are sufficiently
well-distributed.
\begin{theorem}\label{thm:convobj}
  For~$N\in\N$, let
  $0\le\tau_1^{N}<\dots<\tau_N^N< T$
  and~$c_1^{N},\dots,c_N^N\in\{0,1\}$. Define
  $$\Delta \tau^{N}:=
  \max_{j=1,\ldots,N+1}|\tau_j^{N}-\tau_{j-1}^{N}|\;,$$
  where~$\tau_0^N:=0$ and~$\tau_{N+1}^N:=T$.
  If $\Delta \tau^{N} \to 0$
  for~$N\to\infty$, then
  \begin{itemize}
  \item[(i)] the diameters of the feasible sets of~\eqref{eq:SPC}
    and~\eqref{eq:SP} in~$L^2(0,T)$ vanish and
  \item[(ii)] the optimal values of~\eqref{eq:SPC}
    and~\eqref{eq:SP} converge to each other.
  \end{itemize}
\end{theorem}
\begin{proof}
  Let $D_\SP^{N}:=\{u\in D: u(\tau_j^{N}) = c_j^{N} \ \forall
  j=1,\ldots,N\}$ denote the feasible set of~\eqref{eq:SP}
  for~$N\in\N$.  Without loss of generality, we may assume
  $D_\SP^N\neq \emptyset$ for $N\in \N$, since otherwise the feasible
  set of \eqref{eq:SPC} is also empty and thus both optimal values
  agree. We first claim that two controls~$u_1,u_2\in D_\SP^{N}$ can
  only differ in at most~$\sigma$ of the
  intervals~$(\tau_{j-1}^{N},\tau_j^{N})$ for~$2\leq j\leq N$, where
  $\sigma$ denotes the upper bound on the total number of switchings
  guaranteed by Assumption~\eqref{eq:D1}. Indeed, assume that~$u_1$
  and~$u_2$ differ between~$\tau_{j-1}^{N}$ and~$\tau_j^{N}$. Since
  the values of~$u_1$ and~$u_2$ agree at~$\tau_{j-1}^{N}$ and
  $\tau_j^{N}$, either one of the two functions has to switch at least
  twice in~$(\tau_{j-1}^{N},\tau_j^{N})$, if $c_{j-1}^{N}=c_j^{N}$, or
  both functions have to switch at least once, if $c_{j-1}^{N}\neq
  c_j^{N}$. Hence, for each interval where~$u_1$ and~$u_2$ differ,
  both functions together have at least two switchings, but the total
  number of their switchings is bounded by~$2\sigma$.

  Taking into account also the intervals $(0,\tau_1^{N})$
  and~$(\tau_N^{N},T)$ and using that~$u_1,u_2\in [0,1]$
  a.e.\ in~$(0,T)$, we thus obtain
  $$\sup_{u_1,u_2\in D_\SP^{N}} \|u_1-u_2\|^2_{L^2(0,T)} \leq (\sigma+2)\Delta\tau^{N}$$
  and consequently, for~$N\to \infty$, we get
  \begin{equation}\label{eq:vanish}
    \sup_{u_1,u_2\in \overline\conv(D_\SP^{N})} \|u_1-u_2\|_{L^2(0,T)} =
    \sup_{u_1,u_2\in D_\SP^{N}} \|u_1-u_2\|_{L^2(0,T)}
    \to 0\;,
  \end{equation}
  which shows assertion~(i).
  
We now show that the difference $|J(Su_1,u_1)-J(Su_2,u_2)|$ in the objective function vanishes if the difference of the control vanishes. For that, we have a closer look at the solution mapping $S\colon u \mapsto y$ in~\eqref{eq:SP} given by $ S = \Sigma \circ \Psi + \zeta$; see Section~\ref{sec:solmap}. It is well known, see, \eg \cite{Evans_PDE}, that the solution $y=\Sigma w$ satisfies
$$\max_{t\in (0,T)} \|y(t)\|_{L^2(\Omega)} + \|y\|_{L^2(0,T;H_0^1(\Omega))} + \|\partial_t y\|_{L^2(0,T;H^{-1}(\Omega))}\leq C_1 \|w\|_{L^2(0,T;H^{-1}(\Omega))}$$
with a constant~$C_1>0$. For $u_1,u_2\in L^2(0,T)$ we thus obtain 
$$\begin{aligned}
\|Su_1-Su_2\|_{L^2(Q)} = \|\Sigma \Psi u_1 -\Sigma \Psi u_2\|_{L^2(Q)} &\leq \|\Sigma \Psi(u_1 -u_2)\|_{L^2(0,T;H_0^1(\Omega))} \\
&\leq  C_1\|\Psi(u_1 -u_2)\|_{L^2(0,T;H^{-1}(\Omega))} \\
&\leq C_1  \|\psi\|_{H^{-1}(\Omega)}\|u_1-u_2\|_{L^2(0,T)}\;,
\end{aligned}
$$
and hence
$$\begin{aligned}
&|J(Su_1,u_1)-J(Su_2,u_2)|\\&\quad=\tfrac{1}{2}\,\left|\|Su_1 - y_{\textup{d}}\|_{L^2(Q)}^2 -\|Su_2 - y_{\textup{d}}\|_{L^2(Q)}^2  + \alpha\,\|u_1-\tfrac 12\|^2_{L^2(0,T)}- \alpha\,\|u_2-\tfrac 12\|^2_{L^2(0,T)}\right|\\
  &\quad\leq\tfrac{1}{2}\,\left(\|Su_1 -Su_2\|_{L^2(Q)}^2+\alpha \|u_1-u_2\|^2_{L^2(0,T)}\right)\\
  &\quad\leq C_2 \|u_1-u_2\|^2_{L^2(0,T)}
\end{aligned}$$
for some constant~$C_2>0$. Together with~\eqref{eq:vanish}, this
implies that the maximal difference of all objective values of
feasible controls in~\eqref{eq:SPC} vanishes
for~$N\to\infty$. Since~\eqref{eq:SPC} is a relaxation
of~\eqref{eq:SP}, we obtain~(ii).
\end{proof}

As a consequence of Theorem~\ref{thm:convobj} and its proof, we immediately obtain the following.
\begin{corollary}
  For each~$\varepsilon>0$ there exist~$N\in\N$ and
  fixings~$(\tau_j,c_j) \in
  (0,T)\times\{0,1\}$, $j=1,\ldots,N$, such that the optimal value of
  $\eqref{eq:SPC}$ differs by at
  most~$\varepsilon$ from the optimal value of the original
  problem~\eqref{eq:optprob}.
\end{corollary}
In other words, up to an arbitrary desired precision~$\varepsilon>0$, the optimal solution
of~\eqref{eq:optprob} can be approximated by \eqref{eq:SPC} using a finite number of
fixings. This is crucial for the branch-and-bound algorithm we are
going to present in the following. Clearly, the number of necessary
fixings depends on~$\varepsilon$.

To solve the subproblems in the branch-and-bound
algorithm, we will use the outer approximation approach presented in \cite{partI,partII}. For this purpose, we need to discuss how to deal with the
resulting projections under fixings (see Section~\ref{sec:convhull})
and how to adapt the outer approximation algorithm (see
Section~\ref{sec:comp}). For both tasks, first note that
$$\overline\conv(D_\SP)=\overline\conv(\overline{D_\SP})\;,$$ again
with all closures taken in $L^p(0,T)$. Problem~\eqref{eq:SPC} is thus
very similar to the problem without fixings addressed
in~\cite{partI,partII}, except that~$D$ is now replaced by the more
complex set~$\overline{D_\SP}$.  In an outer approximation approach,
the impact of the fixings is then implicitly modeled by the cutting
planes describing~$\overline\conv(D_\SP)$.

However, the fixings may also directly determine significant parts of
the switching pattern in such a way that $u$ must be constant on some
intervals $[\tau_{j-1},\tau_j)$, \ie $u|_{[\tau_{j-1},\tau_j)}\equiv
    c_{j-1}$ for all controls~$u\in D_\SP$. Indeed, as shown by the
    proof of Theorem~\ref{thm:convobj}, the non-fixed part of the time
    horizon vanishes under the assumptions of
    Theorem~\ref{thm:convobj} when~$N\to\infty$. In our
    branch-and-bound algorithm, it is much more efficient to deal with
    these constraints explicity, instead of modeling them by cutting
    planes describing~$\overline\conv(D_\SP)$.
    
Finally, note that it is also possible that
the given fixings are inconsistent with the constraint~$D$, i.e., that the
feasible set of~\eqref{eq:SP} is empty, which is easy to detect for
most choices of~$D$. In this case, the subproblem is infeasible and
the corresponding node in the branch-and-bound tree can be pruned.

\begin{example}\label{ex:fix1}
  Consider the set~$D(\sigma)$ defined
  in~\eqref{eq:dmax}. Let~$N'=|\{j\in\{2,\dots,N\}\colon c_{j-1}\neq
  c_{j}\}|$. If~$N'>\sigma$, we have~$D(\sigma)_\SP=\emptyset$, since
  even the number of switchings enforced by the fixing is too large
  for a feasible solution. The subproblem can thus be pruned.
  If~$\sigma-1\le N'\le\sigma$, we can fix all
  intervals~$[\tau_{j-1},\tau_j)$ with~$c_{j-1}=c_j$ to the
    value~$c_{j-1}$, since any other value in this interval would
    increase the number of switchings by two.  If~$\sigma-1\le
    N'\le\sigma$ and~$c_{j-1}\neq c_j$, then no value of~$u$
    in~$(\tau_{j-1},\tau_j)$ is fixed, but~$u$ has to be monotone
      in~$[\tau_{j-1},\tau_j]$, which is modeled implicitly by cutting
        planes. The same is true for all further restrictions resulting from the
        fixings.  \qed
\end{example}
\begin{example}\label{ex:fix2}
  For the minimum dwell time constraints $D(s)$ defined
  in~\eqref{eq:Ds}, we can fix an interval~$[\tau_{j-1},\tau_j)$
    with~$c_{j-1}=c_j$ to the value~$c_{j-1}$ if and only if
    $\tau_j-\tau_{j-1}\leq s$. Otherwise, no direct fixing is
    possible, but the number of allowed switchings within the
    interval~$(\tau_{j-1},\tau_j)$ reduces to $\lceil
      \nicefrac{\tau_j-\tau_{j-1}}{s}\rceil$. An infeasible subproblem
      arises whenever~$u$ is fixed to the same value at two time
      points having a distance of at most~$s$, but fixed to the other
      value at some point in between. \qed
\end{example}

\section{Convex hull under fixings}\label{sec:convhull}

As already indicated, our aim is to fully describe the convex hull of
feasible switching patterns, i.e., the feasible set of~\eqref{eq:SPC},
by cutting planes derived from finite-dimensional projections,
extending the approach proposed in~\cite{partI} for the case without
fixings. For this, we project the set $\overline{D_\SP}$ to the
finite-dimensional space $\R^M$, by means of local averaging
\begin{equation}\label{eq:Pi}
\Pi\colon BV(0,T)
\ni u \mapsto \Big(\tfrac{1}{\lambda(I_i)}\int_{I_i} u(t) \,\d t\Big)_{i=1}^M\in \R^{M}\;,
\end{equation}
where $I_i\subseteq (0,T)$ for $i=1,\ldots,M$ are suitably chosen
subintervals. Each projection $\Pi$ then gives rise to a relaxation
$$\overline\conv(\overline{D_\SP}) \subseteq \{ v \in L^p(0,T)\colon
\Pi(v) \in C_{\overline{D_\SP},\Pi}\}$$
of the feasible region
\cite[Lemma 3.4]{partI}, where $C_{\overline{D_\SP},\Pi} :=
\conv\{\Pi (u)\colon u \in \overline{D_\SP}\}$. By
a suitable construction of projections $\Pi_k$, with increasing
dimension~$M_k$, a complete outer description of the
finite-dimensional convex hulls $C_{\overline{D_\SP},\Pi}$ also yields a complete outer description of the convex hull of
$\overline{D_\SP}$ in function space \cite[Thm.~3.5]{partI}, \ie
$$\overline\conv(\overline{D_\SP}) =\bigcap_{k\in \N} \{ v \in
L^p(0,T)\colon \Pi_k(v) \in C_{\overline{D_\SP},\Pi_k}\}\;.$$
 
In order to solve the convexified subproblem \eqref{eq:SPC} by means
of the outer approximation algorithm presented
in~\cite[Alg.~1]{partII}, it is particularly desirable that the
sets~$C_{\overline{D_\SP},\Pi}$ are polyhedra for which the separation
problem is tractable, in order to efficiently generate cuts of the
form $a^\top\Pi(u)\leq b$ for $u\in L^p(0,T)$, where $a^\top w\leq b$,
$a\in\R^M$, $b\in \R$ represents a valid inequality for
$C_{\overline{D_\SP},\Pi}$. For prominent examples of~$D$, it is shown
in~\cite{partI} that this is the case for the sets~$C_{D,\Pi}$, \ie
when no fixings are considered. However, it can be shown that the
fixings may destroy this property in general.

For the remainder of this section, we thus focus on the two classes of
constraints~$D$ already discussed in \cite{partI} and defined in
Section~\ref{sec:switchingcon}. We will show that the
sets~$C_{\overline{D_\SP},\Pi}$ are still polyhedra in these cases and
discuss their tractability. For this, we now consider a fixed
projection~$\Pi$ and assume that the intervals~$I_i$,
$i=1,\dots,M$, are pairwise disjoint. Moreover, without loss of
generality, we may assume that the fixing points $0\leq\tau_1< \cdots<
\tau_N < T$ satisfy $\tau_j\notin I_i$ for all~$j=1,\ldots,N$
and~$i=1,\ldots,M$, since otherwise one can refine the projection
intervals and thus generate stronger cutting planes
\cite[Thm.~2.2]{partII}.

\subsection{Restricted total variation}\label{sec:dmax}

If the upper bound~$\sigma$ on the total number of switchings is the
only constraint, as in the definition of~$D(\sigma)$
in~\eqref{eq:dmax}, we obtain the following
\begin{theorem}\label{thm:dmax1}
  The set~$C_{\overline{D(\sigma)_\SP},\Pi}$ is a 0/1 polytope. 
\end{theorem}
\begin{proof}
  The proof is similar to the one of~\cite[Thm.~3.8]{partI}. One can
  again show that $C_{\overline{D(\sigma)_\SP},\Pi}$ is the convex
  hull of all projection vectors resulting from feasible controls that
  are constant almost everywhere on each of the intervals
  $I_1,\ldots,I_N$, \ie $C_{\overline{D(\sigma)_\SP},\Pi}=\conv(K)$
  with
  \[
  \begin{aligned}
	K:=\{\Pi(u) \colon & u \in \overline{D(\sigma)_\SP} \text{ and
          for all }i=1,\ldots,M \text{ there exists } w_i\in
        \{0,1\}\\ & \text{with } u(t)\equiv w_i \text{ f.a.a.\ } t\in
        I_i\}\;.
  \end{aligned}
  \]
  From this, the result follows directly, since $K\subseteq
  \{0,1\}^M$. See Appendix \ref{sec:tvbound} for a detailed proof.
\end{proof}
In the remainder of this subsection, we will show that the separation
problem for~$C_{\overline{D(\sigma)_\SP},\Pi}$ can be solved in
polynomial time.  Without any fixings, \ie when
$D(\sigma)_\SP=D(\sigma)$, the set~$K$ defined in the proof of
Theorem~\ref{thm:dmax1} agrees with
$$\Big\{v\in\{0,1\}^M: \sum_{l=2}^M |v_l-v_{l-1}|\leq
\sigma\Big\}\;.$$ For the slightly different setting where $v_1$ is
fixed to zero, it is shown in~\cite{buchheim23} that the separation
problem for $\conv(K)$ can be solved in polynomial time.  It is easy
to see that the separation problem remains tractable also without this
fixing, \ie for $C_{D(\sigma),\Pi}$.
Our aim is now to efficiently reduce the separation problem for
$C_{\overline{D(\sigma)_\SP},\Pi}$ to the separation problem
for~$C_{D(\sigma),\Pi}$. To this end, we extend the vector~$v$ by the
fixings $c_1,\ldots,c_N$. More precisely, for all~$j\in\{1,\dots,N\}$,
let $i_j\in\{0,\ldots,M\}$ be the index such that $b_{i_{j-1}}\leq
\tau_j\leq a_{i_j}$ holds, where~$b_0:=0$. In addition,
define~$E\colon \R^M \to \R^{M+N}$ by
\begin{equation*}
  Ev=(v_1,\ldots,v_{i_1},c_1,v_{i_1+1},\ldots,v_{i_2},c_2,v_{i_2+1},\ldots,v_{i_N},c_N,v_{i_N+1},\ldots,v_M)^\top\;.
\end{equation*}
The desired reduction is based on the following
\begin{lemma}\label{lem:Cmax}
  A vector $v\in \R^M$ belongs to $K$ if and only
  if $Ev$ belongs to
  $${\cal C}:=\Big\{w\in \{0,1\}^{M+N}: \sum_{l=2}^{M+N}|w_l-w_{l-1}|\leq \sigma\Big\}.$$
\end{lemma}
\begin{proof}
  For the first direction, let $v=\Pi(u)\in K$ for some~$u\in
  \overline{D(\sigma)_\SP}$ being constant almost everywhere on
  each projection interval.
  Then there exists a
  sequence $\{u^m\}_{m\in \N}\subseteq D(\sigma)_\SP$ with $u^m\to u$
  in~$L^p(0,T)$ for $m\to \infty$.
  For every $m\in\N$, the control $u^m$ has at most $\sigma$ switchings and satisfies
  $u^m(\tau_j)=c_j$ for $j=1,\ldots,N$, so that we have
  $$\sum_{l=2}^{M+N}
  |E\Pi(u^m)_l-E\Pi(u^m)_{l-1}|\le\sigma\;.$$
  The continuity of $\Pi$ in
  $L^p(0,T)$ yields $v=\Pi(u)=\lim_{m\to \infty} \Pi(u^m)$ and hence
  $$\sum_{l=2}^{M+N}
  |Ev_l-Ev_{l-1}|\le\lim_{m\to\infty}\sum_{l=2}^{M+N}
  |E\Pi(u^m)_l-E\Pi(u^m)_{l-1}| \le\sigma\;,$$
  \ie we have $Ev\in{\cal C}$ as desired.
	
  We next show the opposite direction. So, let $Ev\in{\cal C}$
  for some vector~$v\in \R^M$. In addition, let
  $0=z_0<z_1<\cdots<z_r=T$ include all endpoints of the intervals
  $I_1\ldots,I_M$ and the fixed positions $\tau_1,\ldots,\tau_N$. Construct
  functions $u^m$ for $m\in \N$ such that
  $$\begin{array}{ll}
    u^m(t)=v_{i} &\ \mbox{for }
    t\in[a_i+\tfrac{\lambda(I_i)}{2m},b_i-\tfrac{\lambda(I_i)}{2m})
      \mbox{ and } i=1,\ldots,N\;,\\ u^m(t)=c_j &\ \mbox{for }
      t\in[\tau_j,\tau_j+\tfrac{\varepsilon_j}{2m}) \mbox{ and }
        j=1,\ldots,N\;,
  \end{array}
  $$ where $\varepsilon_j=\min\{|z_i-\tau_j|: i\in\{1,\ldots,r\},
  z_i\neq \tau_j\}>0$.  For points in $(0,T)$ not covered by the above
  intervals, we copy the value of the left neighboring interval. The
  construction is illustrated in Figure~\ref{fig:um}.

  \begin{figure}
    \centering
    \begin{subfigure}[c]{0.9\textwidth}
      \centering
      \begin{tikzpicture}[scale=1.4]
	\draw[->] (0.8,0)--(9.7,0);
	\node at (1.02,0) {$($};
	\node at (1.98,0) {$)$};
	\node at (1.5,0) [below]{\scriptsize$I_i$};
	\draw (2,0.1)--(2,-0.1) node[below] {\scriptsize$\tau_j$};
	\node at (1.98,0) {$)$};
	\node at (3.52,0) {$($};
	\node at (4.98,0) {$)$};
	\node at (4.25,0) [below]{\scriptsize $I_{i+1}$};
	\draw (5.4,0.1)--(5.4,-0.1) node[below] {\scriptsize$\tau_{j+1}$}; 
	\draw (6.2,0.1)--(6.2,-0.1) node[below] {\scriptsize$\tau_{j+2}$}; 
	\draw (7.6,0.1)--(7.6,-0.1);
	\node at (7.4,-0.3) [right] {\scriptsize$\tau_{j+3}$};
	\node at (6.82,0) {$($};
	\node at (7.58,0) {$)$};
	\node at (7.2,0) [below] {\scriptsize$I_{i+2}$};
	\node at (7.62,0) {$($};
	\node at (9.48,0) {$)$};
	\node at (8.45,0) [below] {\scriptsize$I_{i+3}$};
	
	\draw[dotted,gray] (1.15,2)--(1.15,-0.75);
	\draw[dotted,gray] (1.85,2)--(1.85,-0.75);
	\draw[dotted,gray] (2,2)--(2,-0.75);
	\draw[dotted,gray] (2.3,2)--(2.3,-0.75);
	\draw[dotted,gray] (3.775,2)--(3.775,-0.75);
	\draw[dotted,gray] (4.775,2)--(4.775,-0.75);
	\draw[dotted,gray] (5.4,2)--(5.4,-0.75);
	\draw[dotted,gray] (5.52,2)--(5.52,-0.75);
	\draw[dotted,gray] (6.38,2)--(6.38,-0.75);
	\draw[dotted,gray] (6.2,2)--(6.2,-0.75);
	\draw[dotted,gray] (6.92,2)--(6.92,-0.75);
	\draw[dotted,gray] (7.48,2)--(7.48,-0.75);
	\draw[dotted,gray] (7.6,2)--(7.6,-0.75);
	\draw[dotted,gray] (7.885,2)--(7.885,-0.75);
	\draw[dotted,gray] (9.215,2)--(9.215,-0.75);
	
	\draw[thick] (1.15,0.5)--(1.85,0.5) node[midway,above] {\scriptsize$v_{i}$};
	\filldraw[thick] (1.15,0.5) circle (1pt);
	\draw[thick] (2,1.5)--(2.3,1.5) node[midway,above] {\scriptsize$c_j$};
	\filldraw[thick] (2,1.5) circle (1pt);
	\draw[thick](3.775,1.5)--(4.775,1.5) node[midway,above] {\scriptsize$v_{i+1}$};
	\filldraw[thick] (3.775,1.5) circle (1pt);
	\draw[thick] (5.4,0.5)--(5.52,0.5);
	\node[thick] at (5.4,0.5)  [above] {\scriptsize$c_{j+1}$};
	\filldraw[thick] (5.4,0.5) circle (1pt);
	\draw[thick] (6.2,1.5)--(6.38,1.5);
	\node[thick] at (6.2,1.5) [above] {\scriptsize$c_{j+2}$};
	\filldraw[thick] (6.2,1.5) circle (1pt);
	\draw[thick] (6.92,0.5)--(7.48,0.5) node[midway,above] {\scriptsize$v_{i+2}$};
	\filldraw[thick] (6.92,0.5) circle (1pt);
	\draw[thick] (7.6,1.5)--(7.885,1.5) node[midway,above] {\scriptsize$c_{j+3}$};
	\filldraw[thick] (7.6,1.5) circle (1pt);
	\draw[thick] (7.885,0.5)--(9.215,0.5) node[midway,above] {\scriptsize$v_{i+3}$};
	\filldraw[thick] (7.885,0.5) circle (1pt);
	
	\draw[thick,densely dash dot] (1.82,0.5)--(2,0.5);
	\draw[thick,densely dash dot] (2.3,1.5)--(3.775,1.5);
	\draw[thick,densely dash dot] (4.775,1.5)--(5.4,1.5);
	\draw[thick,densely dash dot] (5.52,0.5)--(6.2,0.5);
	\draw[thick,densely dash dot] (6.38,1.5)--(6.92,1.5);
	\draw[thick,densely dash dot] (7.42,0.5)--(7.6,0.5);
      \end{tikzpicture}
      \caption{Construction of the functions $u^m$, $m\in\N$. \label{fig:um}}
    \end{subfigure}\vspace{2em}
    \begin{subfigure}[c]{0.9\textwidth}
      \centering
      \begin{tikzpicture}[scale=1.4]
	\draw[->] (0.8,0)--(9.7,0);
	\node at (1.02,0) {$($};
	\node at (1.98,0) {$)$};
	\node at (1.5,0) [below]{\scriptsize$I_i$};
	\draw(2,0.1)--(2,-0.1) node[below] {\scriptsize$\tau_j$};
	\node at (1.98,0) {$)$};
	\node at (3.52,0) {$($};
	\node at (4.98,0) {$)$};
	\node at (4.25,0) [below]{\scriptsize $I_{i+1}$};
	\draw(5.4,0.1)--(5.4,-0.1) node[below] {\scriptsize$\tau_{j+1}$}; 
	\draw (6.2,0.1)--(6.2,-0.1) node[below] {\scriptsize$\tau_{j+2}$}; 
	\draw (7.6,0.1)--(7.6,-0.1);
	\node at (7.4,-0.3) [right] {\scriptsize$\tau_{j+3}$};
	\node at (6.82,0) {$($};
	\node at (7.58,0) {$)$};
	\node at (7.2,0) [below] {\scriptsize$I_{i+2}$};
	\node at (7.62,0) {$($};
	\node at (9.48,0) {$)$};
	\node at (8.45,0) [below] {\scriptsize$I_{i+3}$};
	
	\draw[thick] (1,0.5)--(2,0.5) node[midway,above] {\scriptsize$v_{i}$};
	\filldraw[thick](1,0.5) circle (1pt);
	\draw[thick] (2,1.5)--(3.5,1.5) node[midway,above] {\scriptsize$c_j$};
	\filldraw[thick] (2,1.5) circle (1pt);
	\draw[thick] (3.5,1.5)--(5.4,1.5) node[midway,above] {\scriptsize$v_{i+1}$};
	\filldraw[thick] (3.5,1.5) circle (1pt);
	\draw[thick] (5.4,0.5)--(6.2,0.5);
	\node[thick] at (5.5,0.5)  [above] {\scriptsize$c_{j+1}$};
	\filldraw[thick] (5.4,0.5) circle (1pt);
	\draw[thick] (6.2,1.5)--(6.8,1.5);
	\node[thick] at (6.3,1.5) [above] {\scriptsize$c_{j+2}$};
	\filldraw[thick] (6.2,1.5) circle (1pt);
	\draw[thick] (6.8,0.5)--(7.6,0.5) node[midway,above] {\scriptsize$v_{i+2}$};
	\filldraw[thick] (6.8,0.5) circle (1pt);
	\draw[thick] (7.6,0.5)--(9.5,0.5) node[midway,above] {\scriptsize$v_{i+3}$};
	\filldraw[thick] (7.6,0.5) circle (1pt);
      \end{tikzpicture}
      \caption{The limit $u$ of the constructed sequence $\{u^m\}_{m\in\N}$. \label{fig:u}}
    \end{subfigure}
    \caption{Illustration of the second part of the proof of Lemma~\ref{lem:Cmax}.}
  \end{figure}
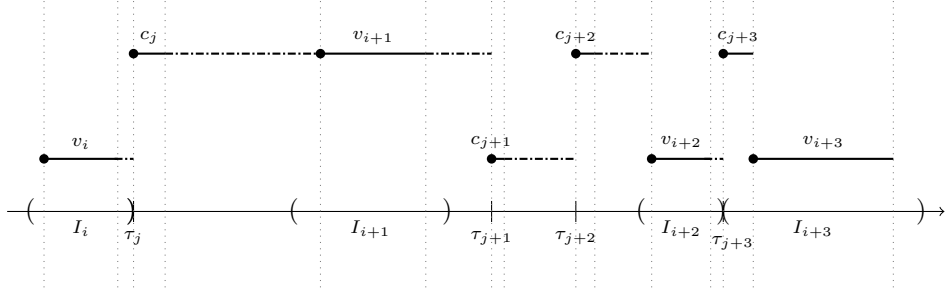
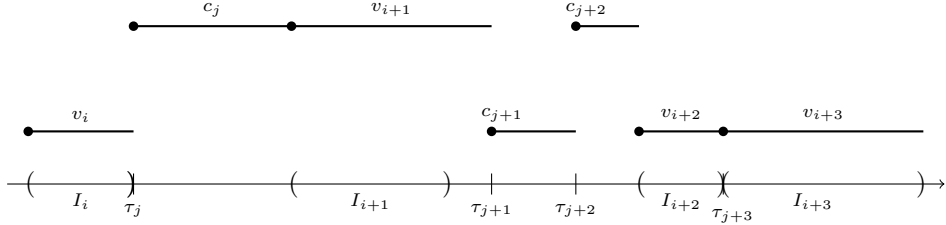

We have $u^m(\tau_j)=c_j$ for every $m\in\N$ and $j=1,\ldots,N$, hence all fixings are respected. Moreover, $Ev\in {\cal C}$ guarantees that $u^m$ switches at most $\sigma$ times, \ie we get $u^m\in D(\sigma)_\SP$. By copying always the value of the left neighboring interval, we guarantee that the control functions $u^m$ converge in $L^p(0,T)$ to some function~$u$; see Figure~\ref{fig:u}. Moreover, by construction, the limit $u$ is $v_1,\ldots,v_M$ almost everywhere on the projection intervals $I_1,\ldots,I_M$, respectively, and due to $\{u^m\}\subseteq D(\sigma)_\SP$, we have $u\in \overline{D(\sigma)_\SP}$. Therefore $v=\Pi(u)\in K$.
\end{proof}
\begin{theorem}\label{thm:dmax2}
  The separation problem for $C_{\overline{D(\sigma)_\SP},\Pi}$ can be
  solved in polynomial time.
\end{theorem}
\begin{proof}
  By the proof of Theorem~\ref{thm:dmax1}, we have
  $C_{\overline{D(\sigma)_\SP},\Pi}=\conv(K)$. Using
  Lemma~\ref{lem:Cmax}, we obtain that $v\in
  C_{\overline{D(\sigma)_\SP},\Pi} $ if and only if $Ev\in \conv({\cal
    C})$. The separation problem
  for~$C_{\overline{D(\sigma)_\SP},\Pi}$ thus reduces to the
  separation problem for~$\conv({\cal C})$.
\end{proof}
The separation algorithm used in the outer approximation approach devised
in~\cite{partI} even computes the most violated cutting plane. The
same can be done when considering fixings: our aim is thus to find the
most violated cutting plane $(\bar{a},\bar{b})\in \R^{M+1}$ in the set
\[
H_{\overline{D(\sigma)_\SP},\Pi}=\{(a,b)\in [-1,1]^M\times \R: a^\top
w \leq b \ \forall w\in C_{\overline{D(\sigma)_\SP},\Pi}\}
\]
of all valid inequalities for $C_{\overline{D(\sigma)_\SP},\Pi}$,
where $a\in[-1,1]^M$ can be assumed without loss of generality by
scaling. It is easily verified that this aim can be achieved by first computing
the most violated cutting plane $(a,b)\in\R^{M+N+1}$ for $Ev$ in
\[
H_{{\cal C}}=\{(a,b)\in [-1,1]^{M+N}\times \R: a^\top w \leq b \ \forall w\in {\cal C}\}\
\]
and then replacing the $(i_j+j)$\,th variable by the constant~$c_j$
for all~$j=1,\dots,N$, \ie $\bar{a}$ results from $a$ by deleting the $(i_j+j)$\,th variables and $\bar{b}=b-\sum_{j=1}^N a_{i_j+j}\,c_j$. The first task can
again be reduced to the case without fixings.

\subsection{Switching point constraints}\label{sec:spc}

We next investigate the set $D(P)$ modeling affine-linear switching
point constraints, as defined in \eqref{eq:DP}. Theorem~3.12 in
\cite{partI} shows that $C_{D(P),\Pi}$ is a polytope in~$\R^M$ by
considering all possible assignments~$\varphi\colon \{1,\ldots,\sigma\}\to
\{1,\ldots,r\}$ of switching points to
intervals~$[z_i,z_{i-1}]$,
where~$0=z_0< z_1 < \cdots z_{r-1} < z_r=\infty$ includes all end
points of the intervals~$I_1,\ldots,I_M$ defining
$\Pi$. Considering the (potentially empty) polytopes
\[
P_\varphi:=\big\{(t_1,\dots,t_{\sigma})\in P: t_1\le\dots\le
t_{\sigma}, \ z_{\varphi(i)-1}\le t_{i}\le z_{\varphi(i)}\; \forall
1\leq i\leq \sigma \big\}\;,\] it is easy to show that
$\{(t_1,\dots,t_{\sigma})\in P: t_1\le\dots\le
t_{\sigma}\}=\bigcup_\varphi P_\varphi$, thus~$C_{D(P),\Pi}$ is the
convex hull of a finite union of polytopes and hence a polytope
itself.

In the presence of fixings, we use a similar approach, but we may only
consider assignments~$\varphi$ respecting the fixings $(\tau_j,c_j)\in
[0,T)\times \{0,1\}$ for $j=1,\ldots,N$. For this purpose, let
  $-1=z_0< z_1 < \cdots z_{r-1} < z_r=\infty$ include all end points
  of the intervals $I_1,\ldots,I_M$ defining~$\Pi$ as well as the
  fixing points $\tau_j$, $j=1,\ldots,N$. In addition, let $\cZ$ be
  the set of all those maps $\varphi$ that, for $j=1,\ldots,N$, assign
  an even number of $t_i$'s to each interval $(\tau_{j-1},\tau_j]$ with
$c_{j-1}=c_j$ and an odd number to each other interval, where we set $\tau_0:=-1$ and
$c_0:=0$ as the switch is supposed to be off at the beginning.
We now define
$$J:= \{i\in\{1,\ldots,\sigma\}:\exists j\in\{1,\ldots,N\} \text{ s.t. }z_{\varphi(i)-1}=\tau_j \}$$ and
$$\begin{aligned}Q_\varphi:=\{(t_1,\ldots,t_\sigma)\in P:\ &t_1\leq \cdots\leq t_\sigma, \ z_{\varphi(i)-1}\leq t_i \leq z_{\varphi(i)}\;\forall i=1,\dots,\sigma,\\ & z_{\varphi(i)-1}<t_i \; \forall i \in J\}\end{aligned}$$
as well as
$$V_\varphi:=\{u_{t_1,\ldots,t_\sigma}\colon
(t_1,\ldots,t_\sigma)\in Q_\varphi\}$$ for all~$\varphi\in\cZ$. Then the
following holds true.
\begin{lemma}\label{lem:fixDp}
  $$D(P)_\SP=\bigcup_{\varphi\in \cZ} V_\varphi\;.$$
\end{lemma}
\begin{proof}
  The proof mainly consists in showing that we can restrict ourselves
  to maps $\varphi \in \cZ$ such that the fixings $(\tau_j,c_j)\in
  [0,T)\times \{0,1\}$ for $j=1,\ldots,N$ are satisfied. The proof can
    be found in Appendix \ref{sec:combswitch}.
\end{proof}

Note that~$\cZ$ is finite, so that
$\overline{D(P)_\SP}=\bigcup_{\varphi\in \cZ}
\overline{V_\varphi}$. Moreover, it can be easily seen that the
closure of each set $V_\varphi$ in $L^p(0,T)$ is given as follows.
\begin{lemma}\label{lem:closevq}
  $$\overline{V_\varphi}= \{u_{t_1,\ldots,t_\sigma}\colon (t_1,\ldots,t_\sigma)\in \overline{Q_\varphi}\}$$
\end{lemma}
\begin{proof}
  See Appendix \ref{sec:combswitch}.
\end{proof}
Now we have everything at hand to prove our main statement. 
\begin{theorem}\label{thm:Dp}
  The set $C_{\overline{D(P)_\SP},\Pi}$ is a polytope. 
\end{theorem}
\begin{proof}
	We have $\Pi(\overline{D(P)_\SP})=\bigcup_{\varphi\in \cZ} \Pi(\overline{V_{\varphi}})$ due to Lemma~\ref{lem:fixDp} and the fact that~$\cZ$ is finite.  
	Since $P$ is a polytope, also~$\overline{Q_\varphi}$ is a polytope. Moreover, analogously to Theorem~3.12 in~\cite{partI}, one
	obtains that the function~$\overline{Q_\varphi}\ni
	(t_1,\dots,t_\sigma)\mapsto\Pi(u_{t_1,\dots,t_\sigma})\in\R^M$ is
	linear for every $\varphi\in \cZ$, so that $\Pi(\overline{V_\varphi})$ is a polytope using Lemma~\ref{lem:closevq}. In summary, we obtain that $\Pi(\overline{D(P)_\SP})$ is a finite union of polytopes and consequently $C_{D(P)_\SP,\Pi}$, as the convex hull of a finite union of polytopes, is a polytope as well.  
\end{proof}

For the remainder of this subsection, we focus on the special case of
dwell time constraints, as defined in \eqref{eq:Ds}. Here, a minimum
time span $s>0$ between two switchings is required. For the case
without fixings, it is stated in~\cite[Thm.~3.14]{partI} that there
exists a separation algorithm with polynomial running time in~$M$ and
in the implicit bound~$\sigma=\lceil \tfrac Ts\rceil$ on the number of allowed switchings. In the presence of
fixings~$(\tau_j,c_j)$, $1\leq j\leq N$, we show in the following that
the separation problem for~$C_{\overline{D(s)_\SP},\Pi}$
is still tractable. More precisely, we claim that there exists a
separation algorithm with polynomial time in~$M$, $\sigma$, and the
number $N$ of fixings.

We thus consider the set
\[
\begin{aligned}
  D(s)_\SP := \big\{ u_{t_1,\dots,t_\sigma}\colon & t_{i}-t_{i-1}\ge s\; \forall \,i = 2, \dots,\sigma,~t_1,\ldots,t_\sigma\geq 0,\\
  & u_{t_1,\dots,t_\sigma}(\tau_j)=c_j\; \forall \,j = 1, \dots,N\big\}.
\end{aligned}
\]
Since~$D(s)$ is a special case of~$D(P)$, the
set~$C_{\overline{D(s)_\SP},\Pi}$ is a polytope in~$\R^M$ by
Theorem~\ref{thm:Dp}. However, it is not a 0/1-polytope in general. In
fact, it is not even a 0/1-polytope without
fixings~\cite[Sect.~3.2]{partI}. Still, the separation problem
for~$C_{\overline{D(s)_\SP},\Pi}$ is tractable. In order to show this,
we use a similar reasoning as in~\cite[Sect.~3.2]{partI} and first
argue that it suffices to consider as switching points the finitely
many points in the set
\[
S:=[0,T]\cap \Big(\Z s+\big(\{0,T\}\cup\{a_i,b_i\colon
i=1,\dots,M\}\cup\{\tau_j\colon
j=1,\dots,N\}\big)\Big)
\] where~$I_i=(a_i,b_i)$ for~$i=1,\dots,M$. The
set~$S$ thus contains, as before, all end points of the
intervals~$I_1,\dots,I_M$ and $[0,T]$ shifted by arbitrary integer
multiples of~$s$, as long as they are included in~$[0,T]$. In
addition, we now need to consider all fixing points
$\tau_1,\ldots,\tau_N$ and their corresponding shiftings. Clearly, we
can compute~$S$ in~$O((M+N)\sigma)$ time.
\begin{lemma}\label{lem:dwell}
  Let~$v$ be a vertex of~$C_{\overline{D(s)_\SP},\Pi}$. Then there
  exists~$u\in \overline{D(s)_\SP}$ with~$\Pi(u)=v$ such that~$u$
  switches only in~$S$.
\end{lemma}
\begin{proof}
  The proof is similar to that of~\cite[Lemma~3.13]{partI}, but one
  needs to pay attention to the fixings when shifting switching points
  outside of~$S$. The full proof can be found in
  Appendix~\ref{sec:combswitch}.
\end{proof}
We next show that there exists an efficient separation algorithm for
$C_{\overline{D(s)_\SP},\Pi}$ by specifying an efficient optimization
algorithm over
$C_{\overline{D(s)_\SP},\Pi}$. Let~$\omega_1\dots,\omega_{|S|}$ be the
elements of~$S$ sorted in ascending order.
\begin{theorem}\label{thm:dwell}
  One can optimize over~$C_{\overline{D(s)_\SP},\Pi}$ (and hence also
  separate from~$C_{\overline{D(s)_\SP},\Pi}$) in time polynomial
  in~$M$, $\sigma$, and $N$.
\end{theorem}
\begin{proof}
  By Lemma \ref{lem:dwell}, it suffices to optimize over the
  projections of all~$u\in \overline{D(s)_\SP}$ with switchings only
  in~$S$. This can be done by a dynamic programming approach similar
  to the one presented in \cite[Thm.~3.14]{partI}; we mainly need to
  change the recursion formula for the fixing
  points~$\tau_1,\ldots,\tau_N$. So assume that an arbitrary
  objective function~$c\in\R^M$ is given. Then we compute the optimal
  value
  \[c^*(t,b):=\min\;c^\top \Pi(u\cdot\chi_{[0,t]})\; \text{
    s.t. }u\in \overline{D(s)_\SP},\;u(t)=b\text{ if }t<T\]
  recursively for all~$t\in S$ as follows.
  Starting with~$c^*(\omega_1,b)=0$ if $\tau_1\neq 0$ and 
  $$c^*(\omega_1,b)=\begin{cases} \infty, &\text{ if } c_1=1 \text{
    and } b=0 \\ 0, &\text{ otherwise}
  \end{cases}$$
  otherwise, we obtain for $k\in\{2,\ldots,|S|\}$ with $\omega_k\in
  S\setminus\{\tau_1,\ldots,\tau_N\}$ that
  \[
  c^*(\omega_k,b)=\min\begin{cases}
  \begin{array}{ll}
    c^*(\omega_{k-1},b)+c^\top\Pi(b\chi_{[\omega_{k-1},\omega_k]})\\[0.75ex]
    \parbox[t]{5cm}{
      $c^*(\omega_k-s,1-b)$\\
      $\text{}\quad+c^\top\Pi((1-b)\chi_{[\omega_k-s,\omega_k]}),$}
    & \parbox[t]{4cm}{
      if $\omega_k\ge s$ and \\
      $\text{}\quad(\omega_k-s,\omega_k)\cap \tau(b)=\emptyset$}\\[3.25ex]
    c^\top\Pi((1-b)\chi_{[0,\omega_k]}),
    & \parbox[t]{4cm}{
      if $\omega_j< s$, $b=1$ and \\
      $\text{}\quad[0,\omega_k)\cap \tau(b)=\emptyset$} 
  \end{array}
  \end{cases}	
  \]
  where for $b\in\{0,1\}$ we define $\tau(b):=\{\tau_j: c_j=b,
  j=1,\ldots,N\}$. For $k\in\{2,\ldots,|S|\}$ with~$\tau_j=\omega_k\in
  S$, $i\in\{1,\ldots,N\}$, we get
  \[
  c^*(\tau_j,c_j)=\min\begin{cases}
  \begin{array}{ll}
    c^*(\omega_{k-1},c_j)+c^\top\Pi(c_j\,\chi_{[\omega_{k-1},\tau_j]})\\[0.75ex]
    \parbox[t]{5cm}{
      $c^*(\tau_j-s,1-c_j)$\\
      $\text{}\quad+c^\top\Pi((1-c_j)\chi_{[\tau_j-s,\tau_j]}),$}
    & \parbox[t]{4cm}{if $\tau_j\ge s$ and \\
      $\text{}\quad(\tau_j-s,\tau_j)\cap \tau(c_j)=\emptyset$} \\[3.25ex]
    0, & \parbox[t]{4cm}{if $\tau_j< s$, $c_j=1$ and \\
      $\text{}\quad[0,\tau_j)\cap \tau(c_j)=\emptyset$} \\
  \end{array}
  \end{cases}	
  \]
  and 
  \[
  c^*(\tau_j,1-c_j)=\begin{cases}
  \begin{array}{ll}
    0,	& \text{if }\tau_j< s \text{ and } c_j=0 \\[0.75ex]
    \infty,  &
    \parbox[t]{5cm}{if $\tau_j <  s$ and $c_j=1$, or 
      $\tau_j\ge s$ and
      $(\tau_j-s,\tau_j)\cap \tau(1-c_j)\neq \emptyset$}\\[3.25ex]
    \parbox[t]{4cm}{$c^*(\tau_j-s,c_j)$ \\
      $\text{}\quad+c^\top\Pi(c_j\,\chi_{[\tau_j-s,\tau_j]}),$}
    &\parbox[c]{4cm}{\text{}\\ otherwise.} 
  \end{array}
  \end{cases}	
  \]
  The desired optimal value is~$\min\{c^*(T,0),c^*(T,1)\}$ then, and a
  corresponding optimal solution can be derived if this value is
  finite. Otherwise, the problem is infeasible due to the fixings, \ie
  the polytope~$C_{\overline{D(s)_\SP},\Pi}$ is empty.
\end{proof}

In the proof of Theorem~\ref{thm:dwell}, the recursion
formula of the dynamic optimization approach over
$C_{\overline{D(s)_\SP},\Pi}$ is the same for the fixing points
$\tau_1,\ldots,\tau_N$ as for the points in $S\setminus
\{\tau_1,\ldots,\tau_N\}$, as long as the fixings are respected. This
is not suprising, since in this case we do not know whether the
control is already constantly $c_j$ before or after $\tau_j$. However,
if the fixing is not respected, then it is clear that the control has
to be constantly $c_j$ on $[\tau_j-s,\tau_j)$ and one has to check
  whether this is possible taking the other fixings and the start
  value zero into account. In particular, if $\tau_1=0$ and $c_1=1$,
  all controls in $\overline{D(s)_\SP}$ have to respect the fixing due
  to the start value zero, so that we have $c^*(0,0)=\infty$ in this
  case.

\section{Computation of primal and dual bounds}\label{sec:comp}

The main task in every branch-and-bound algorithm is the fast
computation of primal and dual bounds. While primal bounds are often
obtained by applying rather straightforward heuristics to the original
problem~\eqref{eq:optprob}, see Section~\ref{sec:primal}, the computation
of dual bounds is a more complex task, see Section~\ref{sec:dual}.

\subsection{Dual bounds}\label{sec:dual}

Our goal is to obtain strong dual bounds by solving the convexified
subproblems~$\eqref{eq:SPC}$; see Section~\ref{sec:branching}. To this
end, we can use the outer approximation algorithm developed
in~\cite{partII},
since~$\overline\conv(D_\SP)=\overline\conv(\overline{D_\SP})$ as
already noted above. This approach is applicable whenever we have a
separation algorithm for~$\overline\conv(\overline{D_\SP})$ at hand;
see Section~\ref{sec:convhull}.  Within the outer approximation
algorithm, we thus need to repeatedly solve problems of the form
\begin{equation}\tag{SPC\text{$_k$}}\label{eq:SPCk}
\left\{\quad		    
\begin{aligned}
\mbox{min }~ &J(Su,u) \\ 
\text{s.t. }~ & u \in [0,1] \quad\text{a.e.~in }(0,T), \\
&  Gu \leq b\;,
\end{aligned}
\qquad\right.
\end{equation}
where $G\colon L^2(0,T)\to \R^k$ with $(Gu)_\ell=a_\ell^\top
\Pi_\ell(u)$ for $\ell=1,\ldots,k$. The latter constraints represent
the cutting planes for the sets $C_{\overline{D_\SP},\Pi}$ that have
been generated so far.

As discussed in Section~\ref{sec:branching}, our branching strategy
will implicitly fix the control~$u$ on certain subintervals of the time horizon~$[0,T]$; see
Example~\ref{ex:fix1} and Example~\ref{ex:fix2}. Let~$\AA$ be the
union of all such fixed intervals and set $\II:=[0,T]\setminus \AA$. Denote
the restrictions to~$\AA$ and~$\II$ by
$\chi_{\AA}:L^2(0,T) \to L^2(\AA)$ and~$\chi_{\II}:L^2(0,T) \to L^2(\II)$, respectively, and let $\chi_\AA^*$
and $\chi_{\II}^*$ be the respective extension-by-zero operators mapping
from $L^2(\AA)$ and $L^2(\II)$ to $L^2(0,T)$, respectively. Then we
can restrict~\eqref{eq:SPCk} to the unfixed control $u|_{\II}$, which leads to
\begin{equation*}
\left\{\quad	
\begin{aligned}
	\mbox{min } &J(S(\chi_{\II}^*u|_{ \II}+\chi_{\AA}^*u|_{ \AA}),\chi_{\II}^*u|_{ \II}+\chi_{\AA}^*u|_{ \AA})=:f(u|_{ \II}) \\ 
	\text{s.t. }~ & u|_{ \II} \in [0,1] \quad\text{a.e.~in } \II, \\
	& G(\chi_{\II}^*u|_{ \II}) \leq b-G(\chi_{\AA}^*u|_{ \AA})\;,
\end{aligned}
\qquad\right.
\end{equation*}
where $u|_{ \AA}$ is fixed and implicitly given through the fixings.
As a first attempt to solve this problem, we applied the semi-smooth
Newton method described in \cite{partII}, but, as the branching
implicitly fixed larger parts of the switching structure, \ie $\AA$
got larger, the semi-smooth Newton method matrix became singular.
To overcome these numerical issues, we decided to replace the
semi-smooth Newton method by the alternating direction method of
multipliers (ADMM), which was first mentioned in~\cite{GM75} for
nonlinear elliptic problems and is widely applied to elliptic control
problems~\cite{AS08,B93,KW18}. Its convergence for convex optimization
problems is well-studied, see, \eg
\cite{FG83,GM76,GLT89,GOS14}. Recently,~\cite{GSY20} also addressed
linear parabolic problems with state constraints by the ADMM and
proved its convergence without any assumptions on the existence and
regularity of the Lagrange multiplier.

We first need to rewrite our
problem in the form
\begin{equation}\tag{SPC\text{$'_k$}}\label{eq:SPCk'}
\left\{\quad	
\begin{aligned}
\mbox{min }~ &f(u|_{ \II})+I_{(-\infty,\bar{b}]}(v)+I_{[0,1]}(w) \\ 
\text{s.t. }~ & u|_{ \II}-w=0 \quad\text{a.e.~in } \II, \\
& G(\chi_{\II}^*u|_{ \II})-v=0\;,
\end{aligned}
\qquad\right.
\end{equation}
where $\bar{b}:= b-G(\chi_{\AA}^*u|_{ \AA})$ and
$$I_{(-\infty,\bar{b}]}(v)=\begin{cases}0, &v\leq \bar{b}  \\
\infty, &\mbox{otherwise}\;, 
\end{cases}\quad \mbox{ and } \quad
	I_{[0,1]}(w)=\begin{cases}0, &w(t)\in[0,1] \mbox{ f.a.a. } t \\
		\infty, &\mbox{otherwise}\;.
	\end{cases} $$
Note that \eqref{eq:SPCk'} is still a convex optimization problem, but no
longer strictly convex. The first-order algorithm~ADMM is an
alternating minimization scheme for computing a saddle point of the
augmented Lagrangian
$$\begin{aligned}
  L_{\rho,\beta}(u|_{\II},v,w,\lambda,\mu) =~ & f(u|_\II) +I_{(-\infty,\bar{b}]}(v)+I_{[0,1]}(w)\\
    &+ \lambda^\top(G(\chi_{\II}^*u|_{\II})-v)+\langle \mu, u|_\II-w\rangle_{L^2(\II),L^2(\II)}\\
    &+ \tfrac \rho 2\|G(\chi_{\II}^*u|_{\II})-v\|^2+\tfrac \beta 2\|u|_{\II}-w\|^2_{L^2(\II)}\;,
\end{aligned}$$
which differs from the Lagrangian by the penalty terms $\tfrac \rho
2\|G(\chi_{\II}^*u|_{\II})-v\|^2$ for the cutting planes and $\tfrac
\beta 2\|u|_{\II}-w\|^2_{L^2(\II)}$ for the box constraints, but has
the same saddle points as the Lagrangian \cite{FG83}. First, the
augmented Lagrangian is minimized with respect to the unfixed control
variables
$$u|_{\II}=\arg \min_{u|_{\II}} L_{\rho,\beta}(u|_{\II},v,w,\lambda,\mu),$$
then with respect to $v$ and $w$,
\ie
$$\begin{aligned}
  v&=\arg \min_{v} L_{\rho,\beta}(u|_{\II},v,w,\lambda,\mu)\;,\\
  w&=\arg \min_w L_{\rho,\beta}(u|_{\II},v,w,\lambda,\mu)\;,
\end{aligned}$$
and finally, the dual variables $\lambda$ and $\mu$ are updated by a
gradient step as follows
$$\begin{aligned} \lambda&=\lambda + \gamma_\rho\rho\,\partial_\lambda L_{\rho,\beta}(u|_{\II},v,w,\lambda,\mu)\;, \\
  \mu&=\mu+\gamma_\beta\beta\,\partial_\mu L_{\rho,\beta}(u|_{\II},v,w,\lambda,\mu)\;.
\end{aligned}$$
For $\gamma_\rho,\gamma_\beta\in (0,\tfrac{1+\sqrt{5}}{2})$, the convergence of ADMM is guaranteed \cite{Glo84}, but these parameters and the penalty parameters influence the convergence performance and numerical stability of the algorithm. For instance, the penalty parameter $\beta$ should be chosen close to $\alpha$ in order to balance the Tikhonov term $\tfrac{\alpha}{2}\|\chi_{\II}^*u|_{\II}+\chi_{\AA}^*u|_{\AA}-\tfrac{1}{2}\|_{L^2(0,T)}$ and the penalty term of the box constraints in the augmented Lagrangian. Moreover, the best choice for $\gamma_\rho$ and $\gamma_\beta$ generally seems to be one \cite{Glo84}. We thus use $\gamma_\rho=\gamma_\beta=1$ in the following.  
 
With the solution mapping $S=\Sigma\circ \Psi +\zeta$, as defined in Section~\ref{sec:pre}, the reduced objective in~\eqref{eq:SPCk'} reads
	$$f(u|_{\II})=\tfrac{1}{2}\, \|\Sigma \Psi(\chi_{\II}^*u|_{\II}+\chi_{\AA}^*u|_{ \AA}) + \zeta - y_{\textup{d}}\|_{L^2(Q)}^2 
	+ \tfrac{\alpha}{2} \,\|\chi_{\II}^*u|_{\II}+\chi_{\AA}^*u|_{ \AA} - \tfrac{1}{2}\|_{L^2(0,T)}^2$$
such that, by the chain rule, its Fr\'echet derivative at $u|_{ \II} \in
L^2(\II)$ is given by
	\begin{equation*}
	f'(u|_{ \II}) = \chi_{\II}\Psi^* \Sigma^* (\Sigma \Psi(\chi_{\II}^*u|_{\II}+\chi_{\AA}^*u|_{ \AA})  + \zeta - y_{\textup{d}}) + \alpha \big(u|_{ \II} - \tfrac{1}{2}\big) \in L^2(\II),
	\end{equation*}
	where we identified $L^2(\II)$ with its dual using the Riesz
	representation theorem. 
	For the penalty term associated with the cutting planes, the Fr\'echet derivative at $u|_{ \II} \in
	L^2(\II)$ is
	$$\rho\, \chi_{\II} G^*\big(G(\chi_{\II}^*u|_{ \II})-v\big)\;.$$ With the above Fr\'echet derivatives at hand, we are able to write down the ADMM method for \eqref{eq:SPCk'}.
Algorithm~\ref{alg:admm} shows the procedure, where $m$ is the iteration counter. 
\begin{algorithm}[htb]
	\caption{ADMM method for \eqref{eq:SPCk'}}\label{alg:admm}
	\begin{algorithmic}[1]
		\STATE{ Choose $v^0,\lambda^0\in\R^\ell\;$, $w^0,\mu^0\in L^2(\II)$ and set $m=0$\label{it:step}}
		\REPEAT
		\STATE{\label{it:3} Solve the equation 
				$$\begin{aligned}(\Psi^*\Sigma^*\Sigma\Psi+(\alpha+\beta)I+\rho\, G^*G)\chi_{\II}^*u_{\ | \II}^{m+1}=~& \Psi^*\Sigma^*\big( y_{\textup{d}}-\zeta-\Sigma\Psi\chi_{\AA}^*u|_{ \AA}\big) -\mu^m+\beta\, w^m\\
				&-G^*\big(\lambda^m-\rho\, v^m\big)+\tfrac{\alpha}{2}\quad  \mbox{a.e. in } \II\end{aligned}$$}
		\STATE{\label{it:4} $v^{m+1}=\min\{ G(\chi_{\II}^*u_{\ | \II}^{m+1})+\tfrac{\lambda^m}{\rho},b-G(\chi_{\AA}^*u|_{ \AA})\}$}
		\STATE{\label{it:5} $w^{m+1}=\max\{\min\{ u_{\ | \II}^{m+1}+\tfrac{\mu^m}{\beta},1\},0\}$}
		\STATE{\label{it:6} $\lambda^{m+1}=\lambda^m + \rho\,\big(G(\chi_{\II}^*u_{\ | \II}^{m+1})-v^{m+1}\big)$}
		\STATE{\label{it:7} $\mu^{m+1}=\mu^m + \beta\,\big(u_{\ | \II}^{m+1}-w^{m+1}\big)$}
		\STATE{$m=m+1$}
		\UNTIL{stopping criterion satisfied}
	\end{algorithmic}
\end{algorithm}

The primal and dual residuals 
$$
r_P^{m}=\begin{pmatrix}
G(\chi_{\II}^*u_{\ | \II}^{m})-v^{m} \\
u_{\ | \II}^{m}-w^{m}
\end{pmatrix}, \quad
r_D^{m}=\rho\, \chi_\II G^*(v^{m-1}-v^{m})+\beta(w^{m-1}-w^{m})
$$
of the optimality conditions for \eqref{eq:SPCk'} can be used to bound the primal objective sub-optimality \cite{BOY11}, \ie $f(u_{\ |\II}^{m})-f(u^\star)$. More precisely, \cite{BOY11} derived sub-optimality estimates for problems in $\R^n$ based on their primal and dual residuals, but the arguments readily carry over to our setting. We thus have
$$ f(u_{\ |\II}^{m})-f(u^\star) \leq - (r_P^{m})^\top \begin{pmatrix}
\lambda^{m} \\ \mu^{m} 
\end{pmatrix} + (u_{\ |\II}^{m}-u_{\ |\II}^\star, r_D^{m})_{L^2(\II)}$$
so that we can estimate 
\begin{equation} \label{eq:errorADMM}
 f(u_{\ |\II}^{m})-f(u^\star) \leq - (r_P^{m})^\top \begin{pmatrix}
\lambda^{m} \\ \mu^{m} 
\end{pmatrix} + \sqrt{T}\, \|r_D^{m}\|_{L^2(\II)}=:e^{m},
\end{equation}
since $u_{\ |\II}^{m},\ u_{\ |\II}^\star \in \{0,1\}$ a.e. in $\II\subset [0,T]$. As a reasonable stopping criterion, we choose that the primal and dual residual must be small, as well as the primal objective sub-optimality. As tolerances for the residuals, we may use an absolute and relative criterion, such as 
$$\begin{aligned}
\|r_P^m\|&\leq (\sqrt{k}+1) \varepsilon^\text{\scriptsize abs} + \varepsilon^\text{\scriptsize rel} \max\{\|G(\chi_{\II}^*u_{\ | \II}^{m})\|_2+\|u^m\|_{L^2(\II)},\|v^m\|_2+\|w^m\|_{L^2(\II)}\}\ , \\ 
\|r_D^m\|&\leq \varepsilon^\text{\scriptsize abs} + \varepsilon^\text{\scriptsize rel} \|\chi_\II G^\star \lambda^m+\mu^m\|_{L^2(\II)}\ ,
\end{aligned}$$
where $\varepsilon^\text{\scriptsize abs}>0$ is an absolute tolerance,
whose scale depends on the scale of the variable values, and
$\varepsilon^\text{\scriptsize rel} > 0$ is a relative tolerance, which
might be $\varepsilon^\text{\scriptsize rel}=10^{-3}$ or
$\varepsilon^\text{\scriptsize rel}=10^{-4}$. The factor $\sqrt{k}$
accounts for the fact that \eqref{eq:SPCk'} contains $k$ cutting plane
constraints. In addition, the absolute error $e^{m}$ in the primal
objective should be less than a chosen tolerance~$\varepsilon^{\text{\scriptsize pr}}>0$.

When the algorithm stops, we obtain $f(u_{\ |\II}^{m})-e^{m}$ as a
dual bound for the subproblem~\eqref{eq:SP} of the branch-and-bound
algorithm, and we can either proceed by calling the separation
algorithm again, in order to generate another violated cutting plane,
if possible, or by stopping the outer approximation algorithm. When
proceeding with the cutting plane algorithm, one has to solve another
parabolic optimal control problem of the form \eqref{eq:SPCk} with an
additional cutting plane $k+1$ by Algorithm~\ref{alg:admm}. The
performance of the algorithm can be improved by choosing the prior
solution $(u,v,\lambda,w,\mu)$ as initialization in
Step~\ref{it:step}, and setting the auxiliary variable to
$v_{k+1}=b-G(\chi_{\AA}^*u|_{ \AA})$ as well as the dual variable to
$\lambda_{k+1}=0$ for the new cutting plane, since the latter is violated by
$u$ for sure.

\subsection{Primal bounds}\label{sec:primal}

Another crucial ingredient in the branch-and-bound framework are
primal heuristics, \ie algorithms for computing good feasible
solutions of the original problem~\eqref{eq:optprob}, which yield
tight primal bounds. It is common to call such primal heuristics in
each subproblem, where the heuristic is often guided by the optimal
solution for the convexified problem being solved in this subproblem
for obtaining a dual bound. In our case, we can apply problem-specific
rounding strategies from the literature to the solution
of~\eqref{eq:SPCk'} found by the ADMM method, \eg the Dwell Time
Sum-up Rounding and Dwell Time Next Force Rounding algorithms~\cite{ZRS20} for the case of a minimum time span between two
switchings, and the Adaptive Maximum Dwell Rounding strategy~\cite{ZS20} for the case of an upper bound on the total number of
switchings.

Moreover, it is often possible to efficiently optimize a linear
objective function over the set~$C_{D,\Pi}$, as shown
in~\cite{partI}. We can benefit from this as follows. First, we define
an appropriate objective function based on the solution~$u$
of~\eqref{eq:SPCk'}. Second, we can use the resulting
minimizer~$v^\star\in C_{D,\Pi}$ and construct a control $u'\in D$
with $\Pi(u')=v^\star$. For the first task, one can consider the distance of
$u$ to $\tfrac{1}{2}$ over the intervals $I_i$ defining the local
averaging operators of the projection $\Pi$ and define the $i$-th
objective coefficient as
\begin{equation}\label{eq:obj}
  \int_{I_i} (\tfrac{1}{2}-u) \,\d t = \lambda(I_i)(\tfrac 12-\Pi(u)_i)\;.
\end{equation}
The intuition in this definition is that a bigger objective
coefficient, \ie a smaller average value of~$u$ on~$I_i$, will promote
a smaller entry~$v^\star_i$ in the minimizer~$v^\star$, and vice
versa. The minimizer~$v^\star$ will thus agree with~$\Pi(u)$ as much
as possible while guaranteeing~$v^\star\in C_{D,\Pi}$. In fact, if
$C_{D,\Pi}$ is a $0/1$-polytope, then the minimization problem
\begin{equation}\label{eq:track}
  \min_{v\in C_{D,\Pi}} ~~ \sum_{i=1}^M \lambda(I_i)~|v_i-\Pi(u)_i|
\end{equation}
can be reformulated as a linear optimization problem over $C_{D,\Pi}$,
which is equivalent to the one with the objective coefficients given
in~\eqref{eq:obj}. Moreover, if the intervals~$I_1,\dots,I_M$ agree
with the given discretization, the minimization
problem~\eqref{eq:track} is equivalent to the CIA~problem addressed
in~\cite{KJS11,JRS15}, which tracks the average of the relaxed
solution over the given temporal grid of the discretization while
respecting the considered switching constraints.
\begin{example}\label{ex:heur1}
  For~$D(\sigma)$, the set~$C_{D(\sigma),\Pi}$ is a $0/1$-polytope~by
  \cite[Thm.~3.8]{partI}, and any linear objective function can be
  optimized in linear time
  over~$C_{D(\sigma),\Pi}$~\cite{buchheim23}. The minimizer~$v^\star$
  can thus be guaranteed to be binary and it can be computed very
  efficiently, which even allows to choose as intervals
  $I_1,\ldots,I_M$ for the projection exactly the ones given by the
  currently used discretization in time. In this case, the minimizer
  $v^\star$ solves the CIA problem over $D(\sigma)$ and it is trivial
  to find a control~$u'$ with~$\Pi(u')=v^\star$. Indeed, on each
  interval~$I_i$, we can set~$u'$ constantly to~$v^\star_i$.  \qed
\end{example}

\begin{example}
  The set $C_{D(s),\Pi}$ of the minimum dwell time constraints is not
  necessarily a $0/1$-polytope, but one can optimize over
  $C_{D(s),\Pi}$ in~$O(M\sigma)$ time, and, by backtracking, one can
  reconstruct the corresponding solution $u'\in D(s)$ in $O(M\sigma)$ time;
  see \cite[Thm.~3.14]{partI}.  \qed
\end{example}
The implicit fixings of the control in a subproblem of the branch-and-bound algorithm can also be considered explicitly in the optimization over $C_{D,\Pi}$ by setting the corresponding objective coefficients in \eqref{eq:obj} to~$\infty$ and~$-\infty$, respectively. More precisely, one may use sufficiently large/small objective coefficients in this case.

In the above examples, a feasible control $u\in D$ can be computed quickly. Nevertheless, in order to obtain the corresponding primal bound, one needs to first calculate the resulting state~$y=S(u)$ and then to evaluate the objective function.

\section{Discretization error and adaptive refinement}\label{sec:numeric}

The dual bounds computed by the outer approximation algorithm
described in the previous section are safe bounds for~\eqref{eq:SPCk},
as long as we do not take discretization errors into account. However,
our objective is to solve~\eqref{eq:optprob} in function space. This
implies that we need to (a) estimate the discretization error
contained in these bounds and (b) devise a method to deal with
situations where the discretization-dependent dual bound allows to
prune a subproblem but the discretization-independent dual bound does
not, \ie where the current primal bound lies between the two dual
bounds. In the latter case, the only way out is the refinement of the
discretization.

In order to address the first task, we will estimate the a posteriori
error of the discretization with respect to the cost functional. We
use the dual weighted residual~(DWR) method, which has already
achieved good results in practice, and combine the results
from~\cite{MV07} and~\cite{VW08} to obtain an error analysis for the
suproblems~\eqref{eq:SPCk} arising in our branch-and-bound tree.
First, we describe the finite element discretization of the optimal
control problems arising in the branch-and-bound algorithm
(Section~\ref{sec:fem}). Then we discuss how to compute safe dual
bounds (Section~\ref{sec:numericdual}) as well as safe primal bounds
(Section~\ref{sec:numericprimal}). Finally, we describe our adaptive
refinement strategy (Section~\ref{sec:refine}).

\subsection{Finite element discretization}\label{sec:fem}

To solve problems of the form~\eqref{eq:SPCk} in practice, we need to discretize the PDE constraint given as 
\begin{equation}\label{eq:state}
\begin{aligned}
\langle \partial_t y, \varphi&\rangle_{L^2(0,T;H^{-1}(\Omega)),L^2(0,T;H_0^1(\Omega))}+(\nabla y,\nabla \varphi)_{L^2(0,T;L^2(\Omega))}+(y(0),\varphi(0))_{L^2(\Omega)}\\&=(\Psi(u),\varphi)_{L^2(0,T;L^2(\Omega))}+(y_0,\varphi(0))_{L^2(\Omega)}\qquad \forall \varphi\in W(0,T)
\end{aligned}
\end{equation}in its weak formulation, as well as the control function, so that we implicitly discretize the Lagrangian~$L\colon W(0,T)\times L^2(0,T)\times W(0,T) \times L^2(0,T)\times L^2(0,T)\times \R^k$ corresponding to~\eqref{eq:SPCk} given as 
$$\begin{aligned}L(y,u,p,\mu^+,\mu^-,\lambda)= ~ J(y,u)&-\langle \partial_t y, \varphi\rangle_{L^2(0,T;H^{-1}(\Omega)),L^2(0,T;H_0^1(\Omega))}-(\nabla y,\nabla p)_{L^2(0,T;L^2(\Omega))}\\
&-(y(0)-y_0,p(0))_{L^2(\Omega)}+(\Psi(u),p)_{L^2(0,T;L^2(\Omega))} \\
&+(\mu^+,u-1)_{L^2(0,T)}- (\mu^-,u)_{L^2(0,T)} + \lambda^\top(Gu-b)\;.
\end{aligned}$$
By calculating the derivative of $L$ w.r.t.~$y$ in arbitrary direction $\varphi\in W(0,T)$, as well as applying interval-wise integration by parts to the equation, we get the adjoint equation 
\begin{equation}\label{eq:adjoint}
\begin{aligned}
-&\langle y, \partial_t p\rangle_{L^2(0,T;H_0^1(\Omega)),L^2(0,T;H^{-1}(\Omega))}+(\nabla\varphi,\nabla p)_{L^2(0,T;L^2(\Omega))}+(\varphi(T),p(T))_{L^2(\Omega)}
\\
&\quad=(\varphi,y-y_\textup{d})_{L^2(0,T;L^2(\Omega))} \qquad \forall \varphi\in W(0,T)\;.
\end{aligned}
\end{equation}

We use a discontinuous Galerkin element method for the time discretization of the PDE constraint with piecewise constant functions. Let 
$$\bar{J}=\{0\}\cup J_1 \cup \cdots\cup J_{L-1}\cup J_L$$ be a partition of $[0,T]$ with half open subintervals $J_l=(t_{l-1},t_l]$ of size~$s_l=t_l-t_{l-1}$ with time points $0=t_0<t_1<\cdots<t_{L-1}<t_L=T$. Define $s:=\max_{l=1,\ldots,N} s_l$ as the maximal length of a subinterval.
The spatial discretization of the state equation uses a standard Galerkin method with piecewise linear and continuous functions, where the domain $\Omega$ is partitioned into disjoint subsets~$K_i$ of diameter~$h_i:=\max_{p,q\in K_i}\|p-q\|_2$ for $i=1,\ldots, R$, \ie $\overline{\Omega}=\cup_{i=1}^R \overline{K_i}$. For the one-dimensional domain $\Omega$ used in our experiments in Section \ref{sec:exp}, this means that we subdivide~$\Omega$ into $R$ disjoint intervals of length $h_i$.  
Set $h:=\max_{i=1,\ldots,R} h_i$ and $\mathcal{K}_h=K_1\cup \cdots\cup K_R$. We define the  finite element space 
$$V_{h}:=\{v\in C(\bar{\Omega})\cap H_0^1(\Omega)\colon v|_K\in P_1(K),\ K\in \mathcal{K}_h\}$$ and associate with each time point $t_l$ a partition $\mathcal{K}_h^l$ of $\Omega$ and a corresponding finite element space $V_h^l\subset H_0^1(\Omega)$ which is used as spatial trial and test space in the time interval $J_l$. Denote by $P_0(J_l,V_h^l)$ the space of constant functions on $J_l$ with values in $V_h^l$. Then we use as a trial and test space for the state equation in~\eqref{eq:optprob} the space
$$X_{s,h}=\{v_{sh}\in L^2(I,L^2(\Omega))\colon v_{sh}|_{J_l}\in P_0(J_l,V_h^l),\ l=1,\ldots,L\}\;.$$ 
By introducing the notation 
$$\begin{aligned}
&y_{sh,l}^+=\lim\limits_{t\searrow 0}y_{sh}(t_l+t), \\
&y_{sh,l}^-=\lim\limits_{t\searrow 0} 
y_{sh}(t_l-t)=y_{sh}(t_l),\; \mbox{and } \\ 
&[y_{sh}]_l:=y_{sh,l}^+ - y_{sh,l}^-\end{aligned}$$ 
for the discontinuities of functions $y_{sh}\in X_{sh}$ in time, 
we obtain the following fully discretized state equation: Find for a given $u_{sh}\in  L^2(0,T)$ a state $y_{sh} \in X_{s,h}$ such that 
\begin{equation}\label{eq:discstate}\begin{aligned}\sum_{l=1}^{L} \langle\partial_t y_{sh}, \varphi\rangle_{J_l}+\sum_{l=1}^{L} &(\nabla y_{sh}, \nabla \varphi)_{J_l}+\sum_{l=1}^{L-1} ([y_{sh}]_l,\varphi_l^+)+(y_{sh,0}^+,\varphi_0^+)\\&=\sum_{l=1}^{L} (\Psi(u_{sh}),\varphi)_{J_l}+ (y_0,\varphi_0^+)\quad \forall \varphi\in X_{s,h}\;,\end{aligned}\end{equation}
where $\langle \cdot,\cdot \rangle_{J_l}:=\langle \cdot,\cdot\rangle_{L^2(J_l;H^{-1}(\Omega)),L^2(J_l;H_0^1(\Omega))}$, $(\cdot,\cdot)_{J_l}:=(\cdot,\cdot)_{L^2(J_l;\Omega)}$, and $(\cdot,\cdot):=(\cdot,\cdot)_{L^2(\Omega)}$. Note that, for piecewise constant states $y_{sh}\in X_{s,h}$, the term $\langle \partial_t y_{sh}, \varphi\rangle_{J_l}$ in \eqref{eq:discstate} is zero for all $l=1,\ldots,L$. 
We denote the discrete solution operator by $S_{sh}:L^2(0,T)\to X_{s,h}$, \ie $y_{sh}=S_{sh}(u_{sh})$ satisfies the discrete state equation~\eqref{eq:discstate} for $u_{sh}\in L^2(0,T)$.
Finally, we use piecewise constant functions for the temporal discretization of the control function on the same temporal grid as for the state equation, \ie we use the space $$Q_\rho=\{w\in L^2(0,T): w|_{J_l}=w_l \mbox{ for all } l=1,\ldots,L\}\;.$$

Altogether, the discretization of~\eqref{eq:SPCk} is given as
\begin{equation}\tag{SPC\text{$_{k\rho}$}}\label{eq:SPCkrho}	\left\{\;		    
\begin{aligned}
\mbox{min }~ &J(y_\rho,u_\rho)\\ 
\text{s.t. }~ & \begin{aligned}\sum_{l=1}^{L} \langle\partial_t y_\rho, &\varphi\rangle_{J_l}+\sum_{l=1}^{L} (\nabla y_\rho, \nabla \varphi)_{J_l}+\sum_{l=1}^{L-1} ([y_\rho]_l,\varphi_l^+)\\&=\sum_{l=1}^{L} (\Psi(u_\rho),\varphi)_{J_l}+(y_0-y_{\rho,0}^+,\varphi_0^+) \quad \forall \varphi\in X_{s,h}\;,\end{aligned} \\
& 0 \leq u_\rho|_{J_l} \leq 1 \quad \mbox{a.e. in } J_l \mbox{ for all } l=1,\ldots,L\;,\\
&  Gu_\rho \leq b \quad \mbox{in } \R^k\;.
\end{aligned}\right.
\end{equation} 
Moreover, the Lagrangian $\tilde{L}\colon X_{s,h}\times Q_\rho\times X_{s,h}\times Q_\rho\times Q_\rho\times \R^k  \to \R$ associated with~\eqref{eq:SPCkrho} results as 
$$\begin{aligned}\tilde{L}(y_\rho,u_\rho,p_\rho,&\mu_\rho^+,\mu_\rho^-,\lambda_\rho)=J(y_\rho,u_\rho)-\sum_{l=1}^{L} \langle\partial_t y_\rho,p_\rho\rangle_{J_l}-\sum_{l=1}^{L} (\nabla y_\rho,\nabla p_\rho)_{J_l}\\&-\sum_{l=1}^{L-1} ([y_\rho]_l,p_{\rho,l}^+)-(y_{\rho,0}^+-y_0,p_{\rho,0}^+)+\sum_{l=1}^{L} (\Psi(u_\rho),p_\rho)_{J_l} \\ 
&+ \sum_{l=1}^{L} \lambda(J_l)(\mu_\rho^+|_{J_l})^\top (u_\rho|_{J_l}-1)- \sum_{l=1}^{L} \lambda(J_l)(\mu_\rho^-|_{J_l})^\top u_\rho|_{J_l} + \lambda_\rho^\top(Gu_\rho-b)\;.
\end{aligned}
$$ 
Based on this, we will devise a posteriori error estimates for both primal and dual bounds in the next sections.

\subsection{A posteriori discretization error of dual bounds\label{sec:numericdual}}

Following the ideas of \cite{MV07,VW08}, we now derive an a posteriori estimate for the error term $J(y,u)-J(y_\rho,u_\rho)$, where $(y,u)\in W(0,T)\times L^2(0,T)$ denotes the optimizer of~\eqref{eq:SPCk} and $(y_\rho,u_\rho)\in X_{s,h}\times Q_\rho$ the one of~\eqref{eq:SPCkrho}.
For this, let us write down the first-order optimality conditions of~\eqref{eq:SPCk} and~\eqref{eq:SPCkrho} by means of the Lagrangian $L$ and $\tilde{L}$, respectively. If $(y,u)\in W(0,T)\times L^2(0,T)$ is optimal for~\eqref{eq:SPCk}, then there exist multipliers~$p\in W(0,T)$, $\mu^+\in L^2(0,T)$, $\mu^-\in L^2(0,T)$ and $\lambda \in \R^k$ such that for $\chi:=(y,u,p,\mu^+,\mu^-,\lambda)$, we have
\begin{subequations}
	\begin{gather}
	L^\prime(\chi)(\delta y,\delta u,\delta p)=0 \quad \forall (\delta y,\delta u,\delta p) \in W(0,T)\times L^2(0,T)\times W(0,T) \label{eq:1a} \\
	\mu^+\geq 0,\quad \mu^+(u-1)=0,\quad u\leq 1 \quad \mbox{a.e. in } (0,T) \label{eq:1b}\\
	\mu^-\geq 0,\quad\quad \mu^-u=0,\quad u\geq 0 \quad \mbox{a.e. in } (0,T) \label{eq:1c}\\ 
	\lambda\geq 0,\quad \lambda^\top (Gu-b)=0, \quad Gu\leq b  \label{eq:1d}
	\end{gather}
\end{subequations}
Analogously, if $(y_\rho,u_\rho)\in X_{s,h}\times Q_\rho$ is optimal for~\eqref{eq:SPCkrho}, then there exist $p_\rho\in X_{s,h}$, $\mu_\rho^+\in Q_\rho$, $\mu_\rho^-\in Q_\rho$ and $\lambda_\rho\in \R^k$ such that for $\chi_\rho:=(y_\rho,u_\rho,p_\rho,\mu_\rho^+,\mu_\rho^-,\lambda_\rho)$ we have
\begin{subequations}
	\begin{gather}
	\tilde{L}^\prime(\chi_\rho)(\delta y,\delta u,\delta p)=0 \quad \forall (\delta y,\delta u,\delta p) \in X_{s,h}\times Q_\rho \times X_{s,h} \label{eq:2a} \\
	\mu_\rho^+|_{J_l}\geq 0,\quad \mu_\rho^+|_{J_l}(u_\rho|_{J_l}-1)=0,\quad u_\rho|_{J_l}\leq 1 \quad \forall l=1,\ldots,L \label{eq:2b}\\
	\mu_\rho^-|_{J_l}\geq 0,\quad\quad \mu_\rho^-|_{J_l}u_\rho|_{J_l}=0,\quad u_\rho|_{J_l} \geq 0 \quad \forall l=1,\ldots,L \label{eq:2c}\\ 
	\lambda_\rho \geq0,\quad \lambda_\rho^\top (Gu_\rho-b)=0,\quad Gu_\rho \leq b \label{eq:2d}
	\end{gather}
\end{subequations}
Using the shorthand notation
$$\begin{aligned}
\mathcal{Y}&=W(0,T)\times L^2(0,T)\times W(0,T)\times L^2(0,T) \times L^2(0,T) \times  \R^k  \mbox{ and } \\
\mathcal{Y}_\rho &= X_{s,h}\times Q_\rho\times X_{s,h}\times Q_\rho\times Q_\rho\times \R^k\;,
\end{aligned}$$ 
we have everything at hand to combine the results from \cite{MV07} and~\cite{VW08} to obtain the following a posteriori discretization error estimation. 
\begin{theorem}\label{thm:E}
	Let $\chi=(y,u,p,\mu^+,\mu^-,\lambda)\in \mathcal{Y}$ fulfill the first-order necessary optimality conditions~\eqref{eq:1a}--\eqref{eq:1d} for~\eqref{eq:SPCk} and $\chi_\rho=(y_\rho,u_\rho,p_\rho,\mu_\rho^+,\mu_\rho^-,\lambda_\rho)\in~\mathcal{Y}_\rho$ the first-order necessary optimality conditions~\eqref{eq:2a}--\eqref{eq:2d} for the discretized problem~\eqref{eq:SPCkrho}.
	Then
	\begin{align*}
	J(y,u)-J(y_\rho,u_\rho)&=\tfrac{1}{2}\tilde{L}^\prime(\chi)(\chi-\chi_\rho)+\tfrac{1}{2}\tilde{L}^\prime(\chi_\rho)(\chi-\chi_\rho) \\ 
	&=\tfrac{1}{2}\Big(\, \tilde{L}_y^\prime(\chi_\rho)(y-y_\rho)+\tilde{L}_p^\prime(\chi_\rho)(p-p_\rho)+\tilde{L}_{u}^\prime(\chi_\rho)(u-u_\rho)\\ &~\qquad+\tilde{L}_{\mu^+}^\prime(\chi)(\mu^+-\mu_\rho^+)+\tilde{L}_{\mu^-}^\prime(\chi)(\mu^--\mu_\rho^-)+\tilde{L}_{\lambda}^\prime(\chi)(\lambda-\lambda_\rho)\\
	&~\qquad+\tilde{L}_{\mu^+}^\prime(\chi_\rho)(\mu^+-\mu_\rho^+)+\tilde{L}_{\mu^-}^\prime(\chi_\rho)(\mu^--\mu_\rho^-)+\tilde{L}_{\lambda}^\prime(\chi_\rho)(\lambda-\lambda_\rho) \,\Big)\;.
	\end{align*}
\end{theorem}
\begin{proof}
	The main arguments of the following proof are taken from the proofs of \cite[Thm.~4.1]{MV07} and \cite[Thm.~4.2]{VW08}. From the first-order optimality system~\eqref{eq:1a}--\eqref{eq:1d} of $\chi\in \cal Y$ for~\eqref{eq:SPCk} we obtain $J(y,u)=L(\chi)$. Analogously, the first-order conditions~\eqref{eq:2a}--\eqref{eq:2d} of $\chi_\rho \in \mathcal{Y}_\rho$ for~\eqref{eq:SPCkrho} lead to $J(y_\rho,u_\rho)=\tilde{L}(\chi_\rho)$.
	Moreover, it holds $L(\chi)=\tilde{L}(\chi)$ since the continuous embedding $W(0,T)\embed C([0,T];L^2(\Omega))$ \cite[Prop.~23.23]{Z90} guarantees $y\in W(0,T)$ to be continuous in time such that the additional jump terms in $\tilde{L}$ compared to $L$ vanish. We thus obtain 
	$$J(y,u)-J(y_\rho,u_\rho)=\tilde{L}(\chi)-\tilde{L}(\chi_\rho)=\int_{0}^1 \tilde{L}^\prime(\chi_\rho+s(\chi-\chi_\rho))(\chi-\chi_\rho) \,\d s\;.$$
	Evaluation of the integral by the trapezoidal rule leads to 
	\begin{equation}\label{eq:*}\tilde{L}(\chi)-\tilde{L}(\chi_\rho)=\tfrac{1}{2}\tilde{L}^\prime(\chi)(\chi-\chi_\rho)+\tfrac{1}{2}\tilde{L}^\prime(\chi_\rho)(\chi-\chi_\rho)+R\end{equation}
	with the residual
	$$R= \tfrac{1}{2} \int_{0}^1  \tilde{L}^{\prime\prime\prime}(\chi + \zeta(\chi - \chi_\rho))(\chi - \chi_\rho,\chi - \chi_\rho,\chi - \chi_\rho)\zeta(\zeta - 1)\,\d\zeta\;.$$
	Since the PDE contained in $\eqref{eq:SPCk}$ as well as the control constraints in $u$ are linear, and the objective is quadratic in $y$ and $u$, respectively, we have $R=0$.
	
	We now have a closer look at the different error terms arising in~\eqref{eq:*}. First, we have
	\begin{equation*}\label{eq:Ea}
	\tilde{L}^\prime(\chi)(\chi-\chi_\rho)=\tilde{L}_{\mu^+}^\prime(\chi)(\mu^+-\mu_\rho^+)+\tilde{L}_{\mu^-}^\prime(\chi)(\mu^--\mu_\rho^-)+\tilde{L}_{\lambda}^\prime(\chi)(\lambda-\lambda_\rho)\;,
	\end{equation*}
	because the other terms are zero thanks to the condition~\eqref{eq:1a}, which can be seen as follows: 
	since $y\in W(0,T)$ is continuous in time due to $W(0,T)\embed C([0,T];L^2(\Omega))$ by \cite[Prop.~23.23]{Z90}, the additional terms in $\tilde{L}_y^\prime$ compared to ${L}_y^\prime$ and $\tilde{L}_p^\prime$ compared to ${L}_p^\prime$, respectively, vanish, so that~\eqref{eq:1a} immediately yields
	$\tilde{L}_y^\prime(\chi)(y)=0$ and $\tilde{L}_p^\prime(\chi)(p)=0$. Moreover, the continuity of $y$ in time implies that 
	$\tilde{L}_p^\prime(\chi)(p_\rho)=0$ can equivalently be expressed as 
	$$\sum_{l=1}^{L} \langle\partial_t y, p_\rho\rangle_{J_l}+\sum_{l=1}^{L} (\nabla y, \nabla p_\rho)_{J_l}+(y_0^+,p_{\rho,0}^+)=\sum_{l=1}^{L} (\Psi(u),p_\rho)_{J_l}+(y_0,p_{\rho,0}^+)\;.$$
	For the continuous state $y$, the state equation \eqref{eq:state} implies that $(\varphi,y(0))=(\varphi,y_0)$ holds for all $\varphi\in L^2(\Omega)$, so that the term $(y_0^+,p_{\rho,0}^+)$ containing $y(0)=y_0^+ $ cancels out with $(y_0,p_{\rho,0}^+)$, as $p_{\rho,0}^+\in L^2(\Omega) ,$ and it remains to ensure  
	$$	\begin{aligned}
	\langle\partial_t y, p_\rho\rangle_{L^2(0,T;H^{-1}(\Omega)),L^2(J_l;H_0^1(\Omega))}+(\nabla y,\nabla p_\rho)_{L^2(0,T;L^2(\Omega))}=(\Psi(u),p_\rho)_{L^2(0,T;L^2(\Omega))}\;. \end{aligned}$$
	Again from the continuous state equation \eqref{eq:state}, the latter equation is satisfied by $y$ such that we obtain $\tilde{L}_p^\prime(\chi)(p_\rho)=0$, as desired.
	It remains to prove $\tilde{L}_y^\prime(\chi)(y_\rho)=0$. To this end, note that $p\in W(0,T)$ is continuous with respect to time by \cite[Prop.~23.23]{Z90}, so that we can rewrite $\tilde{L}_y^\prime(\chi)(y_\rho)=0$ after interval-wise integration by parts in $W(0,T)$ \cite{GGZ75} as 
	\begin{equation*}
	\begin{aligned}
	-\sum_{l=1}^L \langle \partial_t p, y_\rho\rangle_{J_l}+\sum_{l=1}^L(\nabla y_\rho,\nabla p)_{J_l}+(y_{\rho,L}^-,p_{L}^-)
	=\sum_{l=1}^L (y_\rho,y-y_\textup{d})_{J_l}\;.
	\end{aligned}
	\end{equation*}
	Using $p_L^-=p(T)=0$ for the adjoint $p\in W(0,T)$, the above equation becomes 
	\begin{equation*}
		\begin{aligned}
			-\sum_{l=1}^L \langle \partial_t p,y_\rho\rangle_{J_l}+\sum_{l=1}^L(\nabla y_\rho,\nabla p)_{J_l}
			=\sum_{l=1}^L (y_\rho,y-y_\textup{d})_{J_l}\;.
		\end{aligned}
	\end{equation*}
	By the adjoint equation~\eqref{eq:adjoint} and the density of $W(0,T)$ in $L^2(0,T;H_0^1(\Omega))$, the equation is satisfied by $y_\rho\in L^2(0,T;H_0^1(\Omega))$. We thus get $\tilde{L}_y^\prime(\chi)(y_\rho)=0$. 
	Finally,~\eqref{eq:1a} directly yields $\tilde{L}_u^\prime(\chi)(u-u_\rho)=0$ because of~$(u-u_\rho)\in L^2(0,T)$. 
	The second term in \eqref{eq:*} is given as
	\begin{equation*}\label{eq:Eb}
	\begin{aligned}
	\tilde{L}^\prime(\chi_\rho)(\chi-\chi_\rho)=~&\tilde{L}_y^\prime(\chi_\rho)(y-y_\rho)+\tilde{L}_p^\prime(\chi_\rho)(p-p_\rho)+\tilde{L}_{u}^\prime(\chi_\rho)(u-u_\rho)\\&+\tilde{L}_{\mu^+}^\prime(\chi_\rho)(\mu^+-\mu_\rho^+)+\tilde{L}_{\mu^-}^\prime(\chi_\rho)(\mu^--\mu_\rho^-)+\tilde{L}_{\lambda}^\prime(\chi_\rho)(\lambda-\lambda_\rho)\;,
	\end{aligned}
	\end{equation*}
	which completes the proof.
\end{proof}
We need to further specify the estimation of the a posteriori error given in Theorem~\ref{thm:E}, since it contains the unknown solution $\chi\in \mathcal{Y}$. A common approach in the context of the DWR method is to use  higher-order approximations, which work satisfactorily in practice; see, \eg \cite{BR01}. 
Since our control function can only vary over time and the novelty of
our approach lies primarily in the determination of the finitely many
switching points, 
we assume for
simplicity that there is no error caused by the spatial discretization of the state equation to keep the discussion concise. 
Thus, we only use a higher-order interpolation in time. 
For that, we introduce the piecewise linear interpolation operator $I_s^{(1)}$ in time and map the computed solutions to the approximations of the interpolation errors
$$ y-y_\rho\approx I_s^{(1)} y_\rho -y_\rho \mbox{~ and ~} p-p_\rho\approx I_s^{(1)} p_\rho -p_\rho.$$
Then we obtain the approximations
$$\begin{aligned}
\tilde{L}_y^\prime(\chi_\rho)(y-y_\rho)&\approx \tilde{L}_y^\prime(\chi_\rho)(I_s^{(1)} y_\rho -y_\rho)\;, \\
\tilde{L}_p^\prime(\chi_\rho)(p-p_\rho)&\approx \tilde{L}_p^\prime(\chi_\rho)(I_s^{(1)} p_\rho -p_\rho)\;.
\end{aligned}$$
Since the space of the Lagrange multiplier $\lambda$ of the cutting planes is finite-dimensional and thus not implicitly discretized by the discretization of the control space, we may choose $\lambda_\rho$ as higher-order interpolating and consequently neglect the error terms in $\lambda$, \ie  $$\tilde{L}_{\lambda}^\prime(\chi)(\lambda-\lambda_\rho) +\tilde{L}_{\lambda}^\prime(\chi_\rho)(\lambda-\lambda_\rho) \approx 0.$$
Finally, as mentioned in \cite{VW08}, the control $u$ typically does
not possess sufficient smoothness, due to the box and cutting plane
constraints. We thus suggest, as in \cite{VW08}, based on the gradient equation 
$$ L_u^\prime(\chi)=\alpha(u-\tfrac{1}{2})+\Psi^\star p + \mu^+-\mu^- +G^\star \lambda=0$$ and the resulting projection formula 
$$u=\min\{\max\{-\tfrac{1}{\alpha}(\Psi^\star p+ G^\star \lambda)+\tfrac{1}{2},0\},1\}\;,$$ the choice of 
$$\tilde{u}=\min\{\max\{-\tfrac{1}{\alpha}(\Psi^\star I_s^{(1)}p_\rho+ G^\star \lambda_\rho) +\tfrac{1}{2},0\},1\}$$ and $$\tilde{\mu}=-\alpha(\tilde{u}-\tfrac{1}{2})-\Psi^\star I_s^{(1)}p_\rho -G^\star \lambda_\rho=:\tilde{\mu}^+-\tilde{\mu}^-$$ with $\tilde{\mu}^+,\tilde{\mu}^-\geq 0$ a.e. on $(0,T)$. The computable error estimate is thus given as 
\begin{equation}\tag{E\text{$_\eta$}} \label{eq:Eeta}
\begin{array}{rcl}
  \eta & := & J(y,u)-J(y_\rho,u_\rho)\\
  & \approx &\tfrac{1}{2}\Big[
  \tilde{L}_y^\prime(\chi_\rho)(I_s^{(1)} y_\rho -y_\rho) +\
  \tilde{L}_p^\prime(\chi_\rho)(I_s^{(1)} p_\rho -p_\rho)
  +\tilde{L}_{u}^\prime(\chi_\rho)(\tilde{u}-u_\rho)\\
  &&\quad~+\tilde{L}_{\mu^+}^\prime(\tilde{\chi})(\tilde{\mu}^+-\mu_\rho^+)+\tilde{L}_{\mu^-}^\prime(\tilde{\chi})(\tilde{\mu}^--\mu_\rho^-)\\
  &&\quad~+\tilde{L}_{\mu^+}^\prime(\chi_\rho)(\tilde{\mu}^+-\mu_\rho^+)+\tilde{L}_{\mu^-}^\prime(\chi_\rho)(\tilde{\mu}^--\mu_\rho^-)
  \Big]
\end{array}\end{equation}
with
$\tilde{\chi}=(I_s^{(1)}y_\rho,\tilde{u},I_s^{(1)}p_\rho,\tilde{\mu}^+,\tilde{\mu}^-,\lambda_\rho)$.

  As in \cite{MV07}, one could split the error $J(y,u)-J(y_\rho,u_\rho)$
  into (a) the error caused by the semi-discretization of the state equation
  in time, (b) the error caused by the additional spatial discretization
  of the state equation, which we would consider as zero again, and
  (c) the error caused by the control space discretization. This
  would allow to choose different time grids for the state equation
  and the control space, where the former has
  to be at least as fine as the latter
  \cite{MV07}. Since we are mostly interested in the combinatorial switching constraints, so that our focus is on the controls, we decided not to split the error and thus not to consider a finer temporal grid for the state.

  As discussed in Section~\ref{sec:branching}, the given fixings may
determine parts of the switching pattern of~$u$
in~\eqref{eq:SPCk}. In this case, we need to calculate the a posteriori
error~\eqref{eq:Eeta} only on the unfixed control variables
$u|_\II$, as well as on the Lagrange multipliers~$\mu^+,\mu^-\in
L^2(\II)$ corresponding to the box constraints, since we explicitly eliminated the fixed control variables from the problem \eqref{eq:SPCk}. Then, it is clear that the terms $\tilde{L}_{u}^\prime(\chi_\rho)(\tilde{u}-u_\rho)$, $\tilde{L}_{\mu^+}^\prime(\tilde{\chi})(\tilde{\mu}^+-\mu_\rho^+)$, $\tilde{L}_{\mu^-}^\prime(\tilde{\chi})(\tilde{\mu}^--\mu_\rho^-)$, $\tilde{L}_{\mu^+}^\prime(\chi_\rho)(\tilde{\mu}^+-\mu_\rho^+)$, and $\tilde{L}_{\mu^-}^\prime(\chi_\rho)(\tilde{\mu}^--\mu_\rho^-)$ in the error estimator \eqref{eq:Eeta} tend to zero for an increasing number of fixings satisfying the assumptions of Theorem \ref{thm:convobj}, since the non-fixed part of the time horizon vanishes in this case. On the other hand, the error terms $\tilde{L}_y^\prime(\chi_\rho)(I_s^{(1)} y_\rho -y_\rho)$ and
$\tilde{L}_p^\prime(\chi_\rho)(I_s^{(1)} p_\rho -p_\rho)$ reflect the error $J(Su_\rho,u_\rho)-J(y_\rho,u_\rho)$ in the cost functional caused by calculating the discretized state $y_\rho=S_{sh}(u_\rho)$ rather than $Su_\rho$. This error is also taken into account in the primal bounds throughout our branch-and-bound scheme; see Section~\ref{sec:numericprimal} below. 

In summary, in order to numerically compute a safe dual bound for the
subproblem~\eqref{eq:SP}, we first calculate a solution~$u_\rho$  of
the fully discretized problem~\eqref{eq:SPCkrho} with objective value~$J(y_\rho,u_\rho)$ by means of the ADMM method, as described in Section~\ref{sec:dual}. Second, we use $J(y_\rho,u_\rho)-e+\eta$ as a dual bound, where $e$ denotes the absolute error in the primal objective caused by the ADMM algorithm, see~\eqref{eq:errorADMM}, and $\eta$ the a posteriori error of the discretization of~\eqref{eq:SPCk}; compare~\eqref{eq:Eeta}.

\subsection{A posteriori discretization error of primal bounds}\label{sec:numericprimal}

Every feasible solution~$u\in D$, \eg obtained by applying
primal heuristics as described in Section~\ref{sec:primal}, leads to a 
primal bound~$J(Su,u)$ for the original
problem~\eqref{eq:optprob}. However, this bound is again
subject to discretization errors. To estimate the latter, we first
need to solve the fully discretized equation~\eqref{eq:discstate} to
get a state $y_{sh}=S_{sh}(u)$ and then to estimate the a posteriori
error $\nu:=J(Su,u)-J(S_{sh}u,u)$ in the cost functional. For the latter, we can
again use the DWR method, which was originally invented to estimate
the error in the cost function caused by the discretization of the
state equation, see, \eg \cite{BR01}.
We may directly apply \cite[Prop.~2.4]{BR01} to get the approximation 
$$\begin{aligned}
  \nu \approx \ p_y(y_{sh},u,p_{sh})(p-p_{sh}):= &-\sum_{l=1}^L (\nabla y_{sh},p-p_{sh})_{J_l}-\sum_{l=1}^{L-1} ([y_{sh}]_l,p_l^+-p_{sh,l}^+) \\
&-(y_{sh,0}^+-y_0,p_0^+-p_{sh,0}^+)+\sum_{l=1}^L (\Psi(u),p-p_{sh})_{J_l}\;,
\end{aligned}$$
with $\langle \partial_t y_{sh}, p_{sh}\rangle_{J_l}=0$ for $l=1,\ldots,L$, where $p=S^*(y)$ and $p_{sh}=S_{sh}^*(y_{sh})$ denotes the adjoint corresponding to the state $y=S(u)$ and $y_{sh}=S_{sh}(u)$, respectively. Assuming again that there is no error caused by the spatial discretization, we may use the piecewise linear interpolation $I_s^{(1)}p_{sh}$  of $p_{sh}$ in time to obtain the computable a posteriori error 
\begin{equation*}
\nu\approx  p_y(y_{sh},u,p_{sh})(I_s^{(1)}p_{sh}-p_{sh})\;.
\end{equation*}
Then~$J(S_{sh}u,u)+\nu$ is a safe primal bound for~\eqref{eq:optprob}.

\subsection{Adaptive refinement strategy}\label{sec:refine}

The central feature of our branch-and-bound algorithm is the approximate
computation of an optimal solution for~\eqref{eq:optprob} in function
space. In the limit, this solution does \emph{not} depend on any predetermined
discretization of the time horizon. However, in practice, we need to
discretize our subproblems~\eqref{eq:SP} in order to numerically
compute dual bounds, as described in Section~\ref{sec:numericdual}. The main idea of our approach is to use a
coarse temporal grid at the beginning, when the branchings have not
yet determined a significant part of the switching structure, and then
to refine the subintervals (only) if necessary.

More specifically, as long as the time-mesh dependent dual bound
$J(y_\rho,u_\rho)-e$ for~\eqref{eq:SP} is below the best known primal
bound, we proceed with the given discretization. Otherwise, we cannot
find a better solution for~\eqref{eq:SP} for the given
discretization. We then must decide whether better solutions
for~\eqref{eq:SP} may potentially exist when using a finer temporal
grid. This is the case if and only if the time-mesh independent bound
$J(y_\rho,u_\rho)-e+\eta$ is still below the primal bound $PB$. We thus have to refine the grid whenever
$$J(y_\rho,u_\rho)-e+\eta \leq PB < J(y_\rho,u_\rho)-e\;.$$ If even
$J(y_\rho,u_\rho)-e+\eta$ exceeds the primal bound, we can prune the subproblem. Indeed, in this case we cannot
find better solutions for the subproblem even in function space.

The adaptive refinement of the temporal grid is guided by the a posteriori error estimation of the discretization proposed in Section \ref{sec:numericdual}. 
The error estimator \eqref{eq:Eeta} can be easily split into its contribution on each subinterval~$J_l$, \ie 
$$\eta=\sum_{l=1}^L \eta_l,$$
with the local error contributions $\eta_l$ on $J_l$ for $l=1,\ldots,L.$  Note that this splitting is directly possible since we assumed that there is no error caused by the spatial discretization of the state equation, and thus no further localization on each spatial mesh is needed. A popular strategy for mesh adaptation is to order the subintervals according to the absolute values of their error indicators in descending order, \ie to find a permutation $\varrho$ of~$\{1,\ldots,L\}$ such that~$|\eta_{\varrho(1)}| \geq \cdots \geq |\eta_{\varrho(L)}|$,
and then to refine the subintervals which make up a certain percentage $\gamma>0$ of the total absolute error, \ie the subintervals $J_{\varrho(1)},\ldots,J_{\varrho(L_\gamma)}$ with
$$L_\gamma:=\min \Big\{j\in\{1,\ldots,L\}\colon \sum_{l=1}^j |\eta_{\varrho(l)}| > \gamma \sum_{l=1}^L |\eta_l|\Big\}\;.$$

The resulting subproblem \eqref{eq:SPCkrho} with respect to the refined discretization again has to be solved by Algorithm~\ref{alg:admm}. As a reoptimization strategy, the values of the prior discretized solution $(u_\rho,v_\rho,\lambda_\rho,w_\rho,\mu_\rho)$ returned by Algorithm~\ref{alg:admm} can be used to initialize the variables in Step~\ref{it:step}. More precisely, the values of~$(u_\rho,w_\rho,\mu_\rho)$ can be duplicated according to the refinement of the subintervals and~$(v_\rho,\lambda_\rho)$ can be kept unchanged. In this way, we produce a primal feasible solution $(u_\rho,v_\rho,w_\rho)$ for the new subproblem \eqref{eq:SPCkrho}, but note that~$(\lambda_\rho,\mu_\rho)$ is not feasible for the corresponding dual problem.

\section{Numerical experiments} \label{sec:exp}

We now report the results of an extensive numerical evaluation of our branch-and-bound algorithm presented in the previous sections. The overall branch-and-bound method has been implemented in C++, using the \textsc{DUNE}-library~\cite{SAN21} for the discretization of the PDE. The source code can be downloaded at~\url{https://github.com/agruetering/dune-bnb}. 
For all experiments, we discretize the problems as described in Section~\ref{sec:fem}. This means that the spatial discretization uses a standard Galerkin method with continuous and piecewise linear functionals, while the temporal discretization for the control, the state, and the desired state~$y_\textup{d}$ uses piecewise constant functionals in time. The spatial integrals in the weak formulation of the state equation~\eqref{eq:state} and the adjoint equation~\eqref{eq:adjoint}, respectively, are approximated by a Gauss-Legendre rule with order~$3$. This means that all spatial integrals except for the one containing the form function $\varphi$ are calculated exactly. The discretized systems, arising from the discretization of the state and adjoint equation, are solved by a sequential conjugate gradient solver preconditioned with AMG smoothed by SSOR. 
All computations have been performed on a 64bit Linux system with an
Intel Xeon E5-2640 CPU @ 2.5 GHz and $32$ GB RAM.

\subsection{Algorithmic framework}

We start the branch-and-bound algorithm with an equidistant time grid with $20$ nodes
and, if necessary, we refine the subintervals that account for $\gamma=50\,\%$ of the total error; see Section~\ref{sec:refine}.  
The choice of the time point $\tau$ for the branching is crucial for the practical performance of the algorithm, since the implicit restrictions on the controls are highly influenced by the branching points; see Example~\ref{ex:fix1} and Example~\ref{ex:fix2}. Thus, the quality of the dual bounds of each node in the branch-and-bound tree strongly depends on the branching decisions. As already mentioned in Section~\ref{sec:branching}, it is natural to take the last computed relaxed control of the outer approximation algorithm into account, which we know up to a discretization of $(0,T)$; see Section~\ref{sec:fem}.  As a branching point, we choose the point of the time grid where the control has the highest deviation from $0/1$, \ie where the distance to $\{0,1\}$ multiplied by the length of the corresponding grid cell is maximal. This branching strategy corresponds to the choice of the variable with the most fractional value in finite-dimensional integer optimization. Finally, we use breadth-first search as an enumeration strategy since our computed primal bounds track the average of the relaxed solution over the given temporal grid of the discretization, \ie solve the CIA problem over $D(\sigma)$; compare Example~\ref{ex:heur1}. In depth-first search, the shape of the computed relaxed controls for the subproblems hardly changed, so that our primal heuristic always produced the same feasible solution and good primal bounds were found late. As a result, many nodes had to be examined before pruning. This effect is avoided by breadth-first search. 

The results presented in~\cite{partII} suggest to add only a few cutting planes before resorting to branching, because a significant increase in the dual bound was mostly obtained in the first cutting plane iterations. Moreover, we observed that the dual bounds got better with a decreasing Tikhonov parameter~$\alpha$, but the time needed to compute them increased with decreasing~$\alpha$. Thus, we investigate in Section~\ref{sec:param} whether a good quality or a quick computation of the dual bounds have a greater influence on the overall performance. 

The parabolic optimal control problems arising in each iteration of the outer approximation algorithm are solved by the ADMM algorithm; see Algorithm~\ref{alg:admm} in Section~\ref{sec:dual}. As tolerances for the primal and dual residuals in the ADMM algorithm, we chose $\varepsilon^\text{\scriptsize rel}=\varepsilon^\text{\scriptsize abs}=10^{-3}$ and required the absolute error of the discretization of \eqref{eq:SPCk'} to be less than $\varepsilon^\text{\scriptsize pr}=10^{-5}$. In order to guarantee the numerical stability of the ADMM algorithm, the penalty parameter of the cutting planes was set to $\rho=\tfrac{1+\sqrt{5}}{2}$. The best choice of the penalty term $\beta$ of the box constraints depending on the Tikhonov term $\alpha$ is investigated in Section~\ref{sec:param}. The resulting linear system in Step~\ref{it:3} of Algorithm~\ref{alg:admm} is solved by the conjugate gradient method, preconditioned with $P=(\alpha+\beta) I + \rho G^\star G\;.$ 

\subsection{Instances}

 In all experiments, we focus on the
case of an upper bound~$\sigma$ on the number of
switchings, \ie we consider the feasible set
\[
D(\sigma) = \big\{ u \in BV(0,T)\colon \; u(t) \in \{0,1\} \text{ f.a.a.\ } t \in (0,T),\; |u|_{BV(0,T)} \leq \sigma \big\}
\]
as defined in Section~\ref{sec:switchingcon}. However, we assume that $u$
is fixed to zero before the time horizon, so that we already count it
as one switching if~$u$ is $1$ at the beginning. Notwithstanding this
slight modification, the most violated
cutting plane for a given vector $v\notin C_{D(\sigma)_\SP,\Pi}$ can
be computed
in $O(M+N)$ time as discussed in Section~\ref{sec:dmax}, using the
separation algorithm presented in~\cite{buchheim23}. This separation
algorithm is thus fast enough to allow to choose the intervals for the
projection exactly as the intervals given by the discretization in
time; compare Section~\ref{sec:fem}.

We created instances of~\eqref{eq:optprob} with
$\Omega=(0,1)$, $T=1$, and~$\psi(x)=\exp(x)\sin(\pi\,x)+0.5.$ In
order to obtain challenging instances, we produced the desired state~$y_\textup{d}$ as follows: we first generated a control~$u_\textup{d}\colon [0,T] \to \{0,1\}$
with a total variation~$|u_\textup{d}|_{BV(0,T)}=\theta$ and chose
the desired state~$y_\textup{d}$ as~$S(u_\textup{d})$, such that
$u_\textup{d}$ is the optimal solution for Problem~\eqref{eq:optprob} if we allow~$\theta$ switchings. More
specifically, we randomly chose~$\theta$ jump points~$0< t_1< \cdots
<t_{\theta} < T$ on the equidistant time grid with $320$ nodes. Then,
we chose $u_\textup{d}\colon[0,T]\to \{0,1\}$ as the binary control
starting in zero and having the switching points
$t_1,\ldots,t_{\theta}$. In this way, we generated non-trivial
instances, where the constraint $D(\sigma)$ strongly affects the
optimal solution of \eqref{eq:optprob} in case $\sigma\ll \theta$.

\subsection{Parameter tuning}\label{sec:param}

Before testing the potential of our approach, we investigate the influence of some parameters on the overall performance. We first consider the Tikhonov term~$\alpha$ and the penalty term $\beta$ of the box constraints; see Section \ref{sec:dual}. Afterwards, we investigate how time-consuming it is to solve the subproblems arising in the branch-and-bound algorithm, depending on when we stop the outer approximation algorithm for each subproblem \eqref{eq:SP}. Here, we resort to branching if the relative change of the bound is less than a certain percentage (RED) in three successive iterations. Finally, we vary the  allowed relative deviation (TOL) of the objective value of the returned solution from the optimal value of \eqref{eq:optprob}; a subproblem in the branch-and-bound node is pruned when the remaining gap between primal and dual bound falls below this relative threshold. We start with RED\,$=$\,TOL\,$=$\,1\,\%.

For all results presented in this subsection, we have chosen the same instance with $\theta=8$ jump points and allowed $\sigma=3$ switchings, since we observed the typical behavior of the algorithm with these settings. 
We always report the overall number of investigated subproblems~(Subs), of cutting plane iterations~(Cuts), and of ADMM iterations~(ADMM). Moreover, the average number of fixings~($\varnothing$~FixPoints) and the average percentage of control variables that are implicitly fixed~($\varnothing$ FixIndices) are reported, where both averages are taken over all pruned subproblems. We also provide the overall run time (Time) in CPU hours.

\bgroup  \def\arraystretch{1.5}
\begin{table}
	\begin{scriptsize}
		\centering 
		\begin{tabular}{ccrrrrrr}
			\hline
			$\alpha$ & $\beta$  &Subs &  Cuts & ADMM &  $\varnothing$ FixPoints & $\varnothing$ FixIndices &  Time\\
			\hline 0.01& 0.01& 3309 &     6610 &    23489 & 16.07 & 91.65\,\% & 41.91\\
			&0.005 &3253 &     6519 &    19907 & 15.83 & 91.56\,\% & 35.59\\
			&0.001 & 2948 &     5905 &    18889 & 16.52 & 91.25\,\% & 30.84\\
			\midrule  0.005& 0.01 &1961 &     4187 &    17727 & 15.51 & 89.37\,\% & 26.99\\
			&0.005 & 1839 &     3896 &    13588 & 15.06 & 87.45\,\% & 18.33\\
			&0.001 &1764 &     3882 &    17582 & 16.16 & 87.27\,\% & 21.17\\
			\midrule 0.001& 0.01 &1784 &     5076 &    20283 & 17.65 & 87.13\,\% &22.52\\
			&0.005 &1066 &     3400 &     9999 & 14.25 & 81.60\,\% & 10.05\\	
			&0.001 & 1147 &     3426 &    13779 & 13.63 & 81.65\,\% & 13.22\\	
			\hline
		\end{tabular}
		\caption{Influence of the Tikhonov parameter $\alpha$ and the penalty term $\beta$ of the box constraints on the branch-and-bound algorithm.\label{tab:alphabeta}}
	\end{scriptsize}
\end{table}
\egroup

The results for different values of $\alpha$ and $\beta$ can be found in Table~\ref{tab:alphabeta}. The main message of Table~\ref{tab:alphabeta} is that a small value of $\alpha$ is generally favorable for the branch-and-bound algorithm, since a smaller value of $\alpha$ leads to stronger dual bounds and consequently, fewer fixings are needed on average to prune a subproblem. So, as long as no numerical issues arise with the ADMM algorithm and the DWR error estimator, one should choose $\alpha=0.001$. But, with smaller value of $\alpha$ it becomes more likely that the higher-order approximation of the unknown quantities (see Section \ref{sec:numericprimal}) is too imprecise to estimate the error in the cost functional, so that the branch-and-bound algorithm returns wrong solutions. This was also observed in our experiments: in many instances, the obtained solutions for $\alpha \in \{0.01,0.005\}$ switched three times and had very similar switching times for all values of~$\beta$. In contrast, the obtained solutions for $\alpha=0.001$ frequently switched only twice and differed enormously from the others. By recalculating the objective on such a fine grid that all returned solutions are piecewise constant on it, it turned out that the solutions obtained for $\alpha \in \{0.01,0.005\}$ were indeed better than the ones for $\alpha=0.001$. Moreover, the primal heuristic even produced some of the better solutions within the branch-and-bound scheme for $\alpha=0.001$, but due to the DWR error estimator, their time-mesh independent objective values were worse. For that reason, we choose $\alpha=\beta=0.005$ in all subsequent experiments.

\bgroup  \def\arraystretch{1.5}
\begin{table}
	\centering 
	\begin{scriptsize}
		\begin{tabular}{rrrrrrr}
			\hline
			RED &Subs & Cuts & ADMM & $\varnothing$ FixPoints & $\varnothing$ FixIndices & Time\\
			\hline 10\,\% & 1816 &     3610 &    11872 & 15.43 & 88.47\,\% & 17.05\\
			5\,\% & 1821 &     3647 &    11750 & 15.39 & 88.99\,\% & 16.87\\
			2\,\% & 1670 &     3443 &    11940 & 14.31 & 87.91\,\% & 16.86\\
			1\,\% & 1839 &     3896 &    13588 & 15.06 & 87.45\,\% & 18.33\\	
			0.5\,\% & 1857 &     4107 &    14592 & 15.00 & 87.44\,\% &29.25\\
			\hline
		\end{tabular}
		\caption{Impact of the ratio between branching and cutting plane iterations on the branch-and-bound algorithm.\label{tab:objred}}
	\end{scriptsize}
\end{table}
\egroup
We next investigate the interplay between branching and outer approximation. Table~\ref{tab:objred} demonstrates that a good balance is important: a stronger focus on the outer approximation leads to fewer branching decisions needed to cut off a subproblem. However, this does not necessarily imply that fewer fixings are needed to prune a subproblem, since the branching points strongly depend on the shape of the relaxed solutions. Moreover, it is more time-consuming to solve each node due to the increased number of cutting plane iterations. On the other hand, it is also not beneficial to resort to branching too early because more subproblems need to be investigated then. We thus use RED\,$=$\,2\,\% in the following. 

Finally, the impact of the relative allowed deviation from the optimal objective value on the performance of the branch-and-bound algorithm is shown in Table~\ref{tab:tol}. As expected, a higher tolerance leads to an earlier pruning of the subproblems, as indicated by the number of fixings required to prune a subproblem, and the running time decreases significantly. At the same time, however, the best known primal bound (Obj) found by the algorithm obviously increases, so that ultimately the user has to decide which deviation is still acceptable. We choose TOL\,$=$\,2\,\% in the following, which we think is a reasonable optimality tolerance.
\bgroup  \def\arraystretch{1.5}
\begin{table}[]
	\centering
	\begin{scriptsize}
		\begin{tabular}{rrrrrrrr}
			\hline
			TOL & Subs & Cuts & ADMM & $\varnothing$ FixPoints & $\varnothing$ FixIndices &  Obj & Time\\
			\hline 5\,\% &433 & 1123 &     6286 & 9.55 & 73.29\,\%& 0.137512 & 5.53\\
			2\,\% &860 & 1953 & 8644 & 11.66 & 81.62\,\% & 0.135436 & 8.18\\
			1\,\% & 1670 & 3443 & 11940 & 14.31 & 87.91\,\% & 0.135326&16.86\\
			0.5\,\% &3456 & 7437 & 18145 & 17.79 & 93.09\,\% & 0.135214 &50.65\\
			
			\hline 
		\end{tabular}
		\caption{Influence of the relative allowed deviation (TOL) from the optimum on the branch-and-bound algorithm.\label{tab:tol}}
	\end{scriptsize}
\end{table}
\egroup

\subsection{Performance of the algorithm}
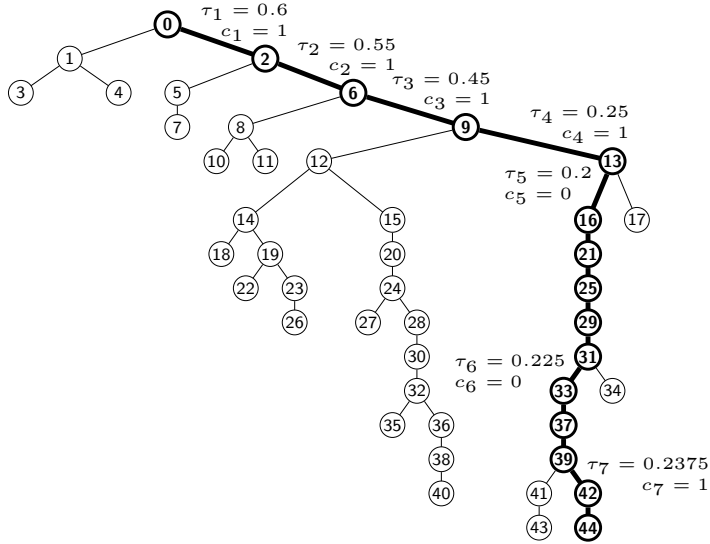
\begin{figure}
	\centering
	\scalebox{1.3}{
		\tiny
		\begin{tikzpicture}
		[sibling distance=20mm,level distance = 3.5mm, line width=0.5mm] \node [arnbold]{0}
		child { [sibling distance=10mm,level distance = 3.5mm,line width=0.1mm] node [arn] {1}
			child { [sibling distance=10mm,level distance = 6mm] node[arn]{3} }
			child { [sibling distance=10mm,level distance = 6mm,black] node [arn] {4} }
		}
		child{ [sibling distance=18mm,level distance = 3.5mm] node [arnbold] {2}
			child { [sibling distance=10mm,level distance = 3.5mm, line width=0.1mm] node[arn]{5} 
				child { [sibling distance=10mm,level distance =6mm, line width=0.1mm] node[arn]{7}}
			}
			child  { [sibling distance=23mm,level distance = 3.5mm] node[arnbold]{6}
				child { [sibling distance=5mm,level distance = 3.5mm, line width=0.1mm] node[arn]{8}
					child { [sibling distance=10mm,level distance = 6mm, line width=0.1mm] node[arn]{10}}
					child { [sibling distance=10mm,level distance = 6mm, line width=0.1mm] node[arn]{11}}
				}
				child { [sibling distance=30mm,level distance = 3.5mm] node[arnbold]{9}
					child { [sibling distance=15mm,level distance = 6mm, line width=0.1mm] node[arn]{12}
						child { [sibling distance=5mm,level distance = 3.5mm, line width=0.1mm] node[arn]{14}	
							child { [sibling distance=10mm,level distance = 6mm, line width=0.1mm] node[arn]{18}}
							child { [sibling distance=5mm,level distance = 3.5mm, line width=0.1mm] node[arn]{19}
								child { [sibling distance=10mm,level distance = 6mm, line width=0.1mm] node[arn]{22}}
								child { [sibling distance=10mm,level distance = 3.5mm, line width=0.1mm] node[arn]{23}
									child { [sibling distance=10mm,level distance = 6mm, line width=0.1mm] node[arn]{26}}
								}
							}
						}
						child { [sibling distance=10mm,level distance =3.5mm, line width=0.1mm] node[arn]{15}
							child { [sibling distance=10mm,level distance = 3.5mm, line width=0.1mm] node[arn]{20}
								child { [sibling distance=5mm,level distance = 3.5mm, line width=0.1mm] node[arn]{24}
									child { [sibling distance=10mm,level distance = 6mm, line width=0.1mm] node[arn]{27}}
									child { [sibling distance=10mm,level distance = 3.5mm, line width=0.1mm] node[arn]{28}
										child { [sibling distance=10mm,level distance = 3.5mm, line width=0.1mm] node[arn]{30}
											child { [sibling distance=5mm,level distance = 3.5mm, line width=0.1mm] node[arn]{32}
												child { [sibling distance=10mm,level distance = 6mm, line width=0.1mm] node[arn]{35}}
												child { [sibling distance=10mm,level distance = 3.5mm, line width=0.1mm] node[arn]{36}
													child { [sibling distance=10mm,level distance = 3.5mm, line width=0.1mm] node[arn]{38}
														child { [sibling distance=10mm,level distance = 6mm, line width=0.1mm] node[arn]{40}}
													}
												}
											}
										}
									}
								}
							}
						}
					}
					child { [sibling distance=5mm,level distance = 6mm] node[arnbold]{13}
						child { [sibling distance=10mm,level distance = 3.5mm] node[arnbold]{16}
							child { [sibling distance=10mm,level distance = 3.5mm] node[arnbold]{21}
								child { [sibling distance=10mm,level distance = 3.5mm] node[arnbold]{25}
									child { [sibling distance=10mm,level distance = 3.5mm] node[arnbold]{29}
										child { [sibling distance=5mm,level distance = 3.5mm] node[arnbold]{31}
											child { [sibling distance=10mm,level distance = 3.5mm] node[arnbold]{33}
												child { [sibling distance=10mm,level distance = 3.5mm] node[arnbold]{37}
													child { [sibling distance=5mm,level distance = 3.5mm] node[arnbold]{39}
														child { [sibling distance=10mm,level distance = 3.5mm, line width=0.1mm] node[arn]{41}
															child { [sibling distance=10mm,level distance = 6mm, line width=0.1mm] node[arn]{43}}
														}	
														child { [sibling distance=10mm,level distance = 3.5mm] node[arnbold]{42}
															child { [sibling distance=10mm,level distance = 6mm] node[arnbold]{44}}
															edge from parent node[near start,right,align=right]{$\tau_7=0.2375$\\ $c_7=1$}
														}
													}
												}
												edge from parent node[near start,left,align=left]{$\tau_6=0.225$\\ $c_6=0$}
											}
											child { [sibling distance=10mm,level distance = 6mm, line width=0.1mm] node[arn]{34}}
										}
									}
								}
							}
							edge from parent node[near start,left,align=left]{$\tau_5=0.2$\\ $c_5=0$}
						}
						child { [sibling distance=10mm,level distance = 6mm, line width=0.1mm] node[arn]{17}}
						edge from parent node[near end,above,align=right]{$\tau_4=0.25$\\ $c_4=1$}								
					}
					edge from parent node[near end,above,align=right]{$\tau_3=0.45$\\ $c_3=1$}				
				} 
				edge from parent node[near end,above,align=right]{$\quad\tau_2=0.55$\\ $\quad c_2=1$}
			}
			edge from parent node[near end,above,align=right]{$\tau_1=0.6$\\ $c_1=1$}
		};
		\end{tikzpicture}}
	\caption{Complete branch-and-bound tree of an instance generated with~$\theta=3$ jump points and with~$\sigma=1$ allowed switchings. The path of the optimal solution is marked in bold and the branching decisions along the optimal path are listed. In the case of a single child node, the temporal discretization of the subproblem has been refined.\label{fig:interplay}}
\end{figure}
Before reporting running times and other key performance indicators of our algorithm, we first illustrate the interplay between branching and adaptive refinement by an example. Figure~\ref{fig:interplay} shows the complete branch-and-bound tree obtained for an instance with $\theta=3$ jump points and only one allowed switching, \ie $\sigma=1$. Whenever a node has a single child node in the illustration, the discretization of the subproblem has been refined. The branch-and-bound tree shows that a large part of the generated subproblems can already be pruned without any refinement. Moreover, in relatively few branches the subproblems need to be refined multiple times in order to decide whether a solution of desired quality can be found in these branches. The branching decisions taken along the path leading to the returned solution illustrate that, \eg the generated subproblem $16$ was refined in order to choose the sixth fixing point as~$\tau_6=0.225$. This was not possible with the previous discretization of the problem. In particular, the fourth and fifth fixing point together have limited the switching point to be in the interval $(0.2,0.25]$. The last branching decision in this tree serves to determine~$t=0.2375$ as the switching point of the returned solution.

\bgroup  \def\arraystretch{1.5}
\begin{table}[!t]
	\centering
	\begin{scriptsize}
		\begin{tabular}{lrrrrrrrrrr}
			\hline
			$\text{}\quad\sigma$& \multicolumn{5}{c}{1} & \multicolumn{5}{c}{2} \\\cmidrule(lr){2-6}\cmidrule(lr){7-11}
			$\theta$&  Subs & Cuts & Time & Refine & Ratio&~~ Subs & Cuts & Time & Refine&Ratio\\ \midrule
			1 &\phantom{00}27.6 & \phantom{00}51.4 &   \phantom{0}0.10  & 3.6 &     7.89\,\%  & & & & & \\
			2 &33.2 & 71.8 &   0.23  & 4.8 &     9.27\,\%  &\phantom{0}157.6 & \phantom{00}292.0 &  \phantom{0}0.74  & 6.6 &     3.59\,\% \\
			3 &32.4 & 69.6 &  0.22  & 3.8 &     5.53\,\%  &132.2 & 274.2 &  1.04  & 4.4 &     9.14\,\%  \\ 
			4 &29.0 & 65.2 &   0.22  &4.0 &    30.75\,\%  &167.2 & 326.0 &  1.02  & 6.8 &     4.45\,\% \\ 
			5&36.4 & 79.2 &   0.20  & 4.2 &     8.58\,\%  &147.6 & 319.4 &  1.04  & 4.6 &     6.46\,\%  \\
			6&18.6 & 49.0 &  0.19  & 1.0 &    64.04\,\%  &202.6 & 410.0 &  1.30  & 5.6 &     2.67\,\% \\ 
			7&32.2 & 75.6 &  0.19  & 2.2 &    25.48\,\%  &247.2 & 518.2 &  1.63  & 4.4 &     2.82\,\% \\
			8&27.0 & 65.6 &   0.23  & 3.0 &    27.88\,\%  &206.2 & 460.2 &  1.49  & 4.6 &     2.99\,\% \\
                        \hline\\
                        \hline
			$\text{}\quad\sigma$& \multicolumn{5}{c}{3} & \multicolumn{5}{c}{4} \\\cmidrule(lr){2-6}\cmidrule(lr){7-11}
			$\theta$&  Subs & Cuts & Time & Refine & Ratio&~~ Subs & Cuts & Time & Refine&Ratio\\ \midrule
			3 & 956.6 & 1848.4 & 8.90  & 7.4 &     \phantom{0}1.86\,\% & &  &  &  &    \\ 
			4 &976.0 & 2128.2 & 8.79  & 7.2 &     1.28\,\%&5572.8 & 11055.6 & 44.29  & 8.0 &     2.09\,\% \\ 
			5&974.0 & 1861.6 & 6.75  & 7.2 &     6.32\,\% &4949.4 & 9194.0 & 43.97  & 7.4 &     2.71\,\% \\
			6&1061.8 & 2278.0 & 10.22  & 7.2 &     1.35\,\% &6255.8 & 12360.8 & 65.06  & 8.0 &     2.44\,\%  \\ 
			7&1239.0 & 2496.2 & 11.15  & 7.2 &     2.41\,\% &6144.6 & 12095.8 & 62.73  & 7.4 &     1.73\,\% \\
			8&1557.2 & 3123.2 & 13.70  & 6.4 &     1.45\,\% &6379.8 & 13005.4 & 66.68  & 7.8 &     5.53\,\% \\ \hline
		\end{tabular}
	\end{scriptsize}
	\caption{Performance of the branch-and-bound algorithm for instances generated with $\theta$ switching points, allowing~$\sigma$ switchings. For each combination of~$\theta$ and~$\sigma$ with $\sigma\leq \theta$, five instances are solved and the average of the number of generated subproblems (Subs), the total cutting plane iterations (Cuts), the total run time in CPU hours (Time), and the maximal number of refinements of a grid cell (Refine) are reported. Moreover, we state the percentage of subproblems (Ratio) whose grid mesh size equals the finest grid mesh size considered.\label{tab:bb}}
\end{table}
\egroup

Table~\ref{tab:bb} shows the performance of the branch-and-bound algorithm for various instances generated with~$\theta\in\{1,\ldots,8\}$ and $\sigma\in\{1,\ldots,4\}$ for the total number of switching points.
We were able to solve problems with up to four allowed switchings, but, as could be expected, the number of generated subproblems strongly increases in~$\sigma$. However, we note that the ratio between generated subproblems and total cutting plane iterations is not affected by $\sigma$. While the branch-and-bound algorithm is able to solve problems with~$\sigma=3$ within 14~CPU hours, the algorithm does not terminate within 60 CPU hours for most instances with~$\sigma=4$ allowed switchings.
However, the results of Table~\ref{tab:bb} show that the average number of subproblems in the branch-and-bound-tree remains relatively small for all instances, showing that the dual bounds computed by our algorithm are rather tight, and that the main challenge in terms of running times is the fast computation of these dual bounds.

Moreover, the reported results show that our approach to globally solve parabolic optimal control problems with dynamic switches by means of branch-and-bound, combined with an adaptive refinement strategy, works in practice. Whenever the maximal number of refinements of a grid cell within the branch-and-bound algorithm was larger than $4$ in our experiments, a grid cell was refined this often in less than $10\,\%$ of the subproblems. Here, the finest grid mesh size decreases with the number of allowed switching points. This means that, if more switchings are allowed, a finer temporal discretization is needed to detect the optimal positions of the switching points.

In summary, our proposed branch-and-bound method is an effective and robust algorithm to globally solve control problems of the form \eqref{eq:optprob}. A few pointwise fixings of the controls suffice to significantly truncate the set of feasible switching patterns. Moreover, thanks to the computation of tight dual bounds by means of outer approximation, relatively few subproblems need to be inspected and refined within the branch-and-bound algorithm.

\appendix

\section{Finite-dimensional convex hulls under fixings}
\subsection{Restricted total variation}\label{sec:tvbound}
We show that if we restrict the total variation of a single switch to be less than $\sigma >0$, \ie the set of feasible switching patterns is given by
\begin{equation*}
D(\sigma)=\{u\in BV(0,T):u(t)\in \{0,1\} \text{ a.e. in }(0,T),\ |u|_{BV(0,T)}\leq \sigma \}\;,
\end{equation*}
then the convex hull $C_{\overline{D(\sigma)_\SP},\Pi}$ of the finite dimensional projections $\{\Pi(u):u\in D(\sigma)_\SP\}$ under arbitrary fixings is a 0/1 polytope. 
\begin{theorem}
	The set~$C_{\overline{D(\sigma)_\SP},\Pi}$ is a 0/1 polytope. 
\end{theorem}
\begin{proof}
	The proof is similar to the one of Theorem 3.8
        in~\cite{partI}, where no fixings have been considered. We
        claim that~$C_{\overline{D(\sigma)_\SP},\Pi}=\conv(K)$, where
	\[
	\begin{aligned}
	K:=\{\Pi(u) \colon & u \in \overline{D(\sigma)_\SP} \text{ and for all }i=1,\ldots,M \text{ there exists } w_i\in \{0,1\}\\
	& \text{with } u(t)\equiv w_i \text{ f.a.a.\ } t\in I_i\}\;.
	\end{aligned}
	\]
	From this, the result follows directly, as~$K\subseteq\{0,1\}^M$
	holds by definition.
	
	Since $K$ is a subset of $\{\Pi(u) \colon u \in
        \overline{D(\sigma)_\SP}\}$, the direction
        \grqq$\supseteq$\grqq\ is trivial.  It thus remains to show
        \grqq$\subseteq$\grqq. For this, let $u\in
        \overline{D(\sigma)_\SP}$. We need to prove that $\Pi(u)$ can
        be written as a convex combination of vectors in~$K$. Let
        $l\in\{0,\dots,M\}$ denote the number of intervals in which
        the switch~$u$ is switched at least once. We prove the
        assertion by means of complete induction over the number
        $l$. For $l=0$, we clearly have $\Pi(u)\in K\subseteq
        \conv(K)$.  So let~$l>0$ and choose an
        index~$\ell\in\{1,\ldots,M\}$ such that~$u$ switches
        at least once in~$I_\ell$. For~$k=0,1$, define the
        function~$u_k$ as follows:
	\[u_k(t):=\begin{cases}
	k,& \text{if }t\in I_\ell \\ 
	u(t), & \text{otherwise.}
	\end{cases}
	\]
	Then, by construction, $\Pi(u)_\ell = \lambda \Pi(u_1)_\ell + (1-\lambda)\Pi(u_0)_\ell $ for $\lambda:=\Pi(u)_\ell\in [0,1]$ and $u_k$ has at most as many switching as $u$.  
	
	We next show that the controls $u_0$ and~$u_1$ belong to $\overline{D(\sigma)_\SP}$. So let $k\in \{0,1\}$ be arbitrary. 
	Due to $u\in \overline{D(\sigma)_\SP}$, there exists a sequence $\{v^m\}_{m\in \N}\in D(\sigma)_\SP$ such that $v^m\to u$ in $L^p(0,T)$ for $m\to \infty$. In particular, there exists a subsequence, which we denote by the same
	symbol for simplicity, with $v^m(t)\to u(t)$ f.a.a. $t\in (0,T)$ for $m\to \infty$. Since $u$ switches at least once in the interval~$I_\ell$ and $v^m$ converges pointwise almost everywhere to $u$, there exists $m_0\in\N$ such that for all $m\geq m_0$ the controls~$v^m$ also switch at least once in $I_\ell$. When constructing a sequence in $D(\sigma)_\SP$ converging to $u_k$ with the help of $\{v^m\}_{m\in \N}$, we need to consider that fixing points $\tau_j$ may coincide with the interval limits of $I_\ell=(a_\ell,b_\ell)$ so that we are only able to change the values in the inner of $I_\ell$. Thus, we define 
	$$w_k^m(t)=\begin{cases}
	k, & t\in [a_\ell+\tfrac{\lambda(I_\ell)}{2m}, b_\ell-\tfrac{\lambda(I_\ell)}{2m}) \\
	v^m(a_\ell), &t\in [a_\ell,a_\ell+\tfrac{\lambda(I_\ell)}{2m}) \\
	v^m(b_\ell), &t\in [b_\ell-\tfrac{\lambda(I_\ell)}{2m},b_\ell) \\
	v^m(t), &\text{otherwise}\;. 
	\end{cases}$$
	Due to $v^m\in \{0,1\}$ a.e.~in $(0,T)$, also $w_k^m(t)\in \{0,1\}$ holds f.a.a.\ $t\in(0,T)$. By our general assumption, we have $\tau_j\notin(a_\ell,b_\ell)$ for all $j=1,\ldots,N$, so that $w_k^m(\tau_j)=v^m(\tau_j)=c_j$ follows with~$v^m\in D(\sigma)_\SP$. Furthermore, for $m\geq m_0$, $w_k^m$ has at most as many switchings as
	$v^m$ in total and we thus obtain $w_k^m\in D(\sigma)_\SP$ for $m\geq m_0$. It is easy to see that $w_k^m \to u_k$ in~$L^p(0,T)$ for $m\to \infty$, so that we get $u_k\in \overline{D(\sigma)_\SP}$, as claimed.   
	
	By the induction hypothesis, the
	vectors~$\Pi(u_k)$ can thus be written as convex combinations of vectors
	in $K$ and consequently, also $\Pi(u)$ is
	a convex combination of vectors in~$K$.
\end{proof}

\subsection{Switching point constraints}\label{sec:combswitch}

In the following, we show some auxiliary results for the class 
\begin{equation*} \begin{aligned}
  D(P) := \{ u_{t_1,\dots,t_\sigma} \colon (t_1, \ldots, t_\sigma) \in P,~
  0\leq t_1\leq \cdots\leq t_\sigma < \infty\}
\end{aligned}
\end{equation*}
of switching point constraints, where $P\subseteq\R_+^\sigma$ is a given polytope. These results are used to show Theorem~\ref{thm:Dp} and Theorem~\ref{thm:dwell} in Section~\ref{sec:spc}, stating that the finite-dimensional convex hulls under fixings are still polytopes and that the corresponding separation problems are tractable in the case that
$P=\{t\in\R_+^\sigma: t_i-t_{i-1}\geq s \ \forall i=1,\ldots,\sigma\}$ for some $s>0$.
Using the notation introduced in Section~\ref{sec:spc}, we first show
\begin{lemma}\label{lem:DPSP}
$$D(P)_\SP=\bigcup_{\varphi\in \cZ} V_\varphi\;.$$
\end{lemma}
\begin{proof}
Let $u=u_{t_1,\ldots,t_\sigma}\in D(P)_\SP$ with~$(t_1,\ldots,t_\sigma)\in P$. Define $\bar{\varphi}: \{1,\ldots,\sigma\}\to \{1,\ldots,r\}$ such that $z_{\bar{\varphi}(i)-1}<t_i\leq z_{\bar{\varphi}(i)}$ holds for $i=1,\ldots,\sigma$. Due to $u_{t_1,\ldots,t_\sigma}(\tau_1)=c_1$, the other fixings $u_{t_1,\ldots,t_\sigma}(\tau_j)=c_j$, $2\leq j\leq N$, can only be satisfied if the number of switching points in $(\tau_{j-1},\tau_{j}]$ is even in the case $c_{j-1}=c_{j}$ and odd, otherwise. If $c_1=0$, then $u_{t_1,\ldots,t_\sigma}(\tau_1)=0$ only holds if an even number of switching points is less or equal to $\tau_1$, and in the other case~$c_1=1$, this number must be odd. Consequently, we obtain $\bar{\varphi}\in \cZ$  and $u\in V_{\bar{\varphi}}$. 

        For the reverse inclusion, let $u\in V_\varphi$ for some $\varphi \in \cZ$. Then there exists $(t_1,\ldots,t_\sigma)\in Q_\varphi$ such that $u=u_{t_1,\ldots,t_\sigma}$. With $Q_\varphi \subseteq P$ it follows that $u\in D(P)$. Since  $\varphi \in \cZ$, we know that the correct number of switching points is assigned between $\tau_{j-1}$ and $\tau_{j}$ in order to respect the given fixings in $D(P)_\SP$ . Moreover,
  the last requirement in the definition of $Q_\varphi$ ensures that no switching point assigned to the right neighboring interval of $\tau_j$ is equal to $\tau_j$, so the given fixings~$u_{t_1,\ldots,t_\sigma}(\tau_j)=c_j$ are indeed satisfied for all~$j\in\{1,\ldots,N\}$, which completes the proof.
\end{proof}
To show that the convex hull of all projection vectors from controls $u\in\overline{D(P)_\SP}$ is a polytope, we can use that $\overline{D(P)_\SP}=\bigcup_{\varphi\in \cZ} \overline{V_\varphi}$ holds, thanks to Lemma~\ref{lem:DPSP} and the fact that $\cZ$ is finite. Consequently, we essentially need that $\Pi(\overline{V_\varphi})$ is a polytope for every $\varphi$ to deduce the polyhedricity of $C_{\overline{D(P)_\SP},\Pi}$; compare Theorem \ref{thm:Dp}. For this, we now prove that we simply need to consider the closure of the sets~$Q_\varphi$ in $\R^\sigma$ to obtain~$\overline{V_\varphi}$, with the help of the following auxiliary result:
\begin{lemma}\label{lem:rep}
	If there exists a sequence $\{u_{t_1^m,\ldots,t_\sigma^m}\}_{m\in \N}$ with $ u_{t_1^m,\ldots,t_\sigma^m}\to u$ in $L^p(0,T)$ for some $u\in L^p(0,T)$ and $t^m:=(t_1^m,\ldots,t_\sigma^m) \to \bar{t}$ in $\R^\sigma$, then $u=u_{\bar t_1,\dots,\bar t_\sigma}$. 
\end{lemma}
\begin{proof}
	The assertion is proven in \cite[Lemma~3.10]{partI} and is based on the continuity of the mapping $\R^\sigma\ni(t_1,\dots,t_\sigma)\mapsto u_{t_1,\ldots,t_\sigma}\in L^p(0,T)$.
\end{proof}
\begin{lemma} 
  $$\overline{V_\varphi}= \{u_{t_1,\ldots,t_\sigma}\colon (t_1,\ldots,t_\sigma)\in \overline{Q_\varphi}\}$$
\end{lemma}
\begin{proof}
  First, let $u\in\overline{V_\varphi}$ and consider a sequence~$\{u^m\}_{m\in \N}$ in~$V_\varphi$ with $u^m =u_{t^m_1,\dots,t^m_\sigma}\to u$ in $L^p(0,T)$, where~$t^m=(t^m_1,\dots,t^m_\sigma)\in Q_\varphi$.
  The strong convergence in $L^p$ implies that there is a subsequence, denoted by the same symbol for convenience, which converges pointwise 
	almost everywhere in $(0,T)$ to~$u$.
        Furthermore, as a polytope, $P$ is bounded by definition, so that~$Q_\varphi$ is bounded as well and thus
	there is yet another subsequence such that
	$t^m$ converges to $\bar t \in \overline{Q_\varphi}$. With Lemma~\ref{lem:rep}, we may conclude that $u_{\bar t_1,\dots,\bar t_\sigma}=u$ and, thanks to 
	$\bar t \in \overline{Q_\varphi}$, this finishes the proof of the first inclusion.
	
	For the reverse inclusion, consider $u_{t_1,\ldots,t_\sigma}$ with switching points~$t=(t_1,\ldots,t_\sigma)\in \overline{Q_\varphi}$. Since $t\in \overline{Q_\varphi}$, there exists a sequence $t^m=(t^m_1,\dots,t^m_\sigma)\in Q_\varphi$ with $t^m\to t$ in $\R^\sigma$. Again thanks to the continuity of the mapping $(t_1,\dots,t_\sigma)\mapsto u_{t_1,\ldots,t_\sigma}$~\cite[Lemma~3.10]{partI}, the sequence~$\{u_{t^m_1,\ldots,t^m_\sigma}\}_{m\in \N}\subseteq V_\varphi$ converges to~$u_{t_1,\ldots,t_\sigma}$ in~$L^p(0,T)$, so that the latter belongs to the closure of~$V_\varphi$ in~$L^p(0,T)$.
\end{proof}
Besides the fact that $C_{\overline{D(P)_\SP},\Pi}$ is a polytope, it is also crucial for our approach that there exists an efficient separation algorithm for this set. Indeed, for the special case
\[
\begin{aligned}
  D(s)_\SP:=\big\{ u_{t_1,\dots,t_\sigma}\colon &t_{i}-t_{i-1}\ge s ~ \forall \,i = 2, \dots,\sigma,~ t_1,\ldots,t_\sigma\geq 0,\\ &u_{t_1,\dots,t_\sigma}(\tau_j)=c_j\; \forall \,j = 1, \dots,N\big\}
\end{aligned}
\]
of dwell time constraints with fixings
$(\tau_j,c_j)\in[0,T)\times\{0,1\}$, $1\leq j\leq N$, the separation
  problem is polynomially solvable in the number $M$ of projection
  intervals, the number $\sigma$ of allowed switchings, and the number
  $N$ of fixings, as claimed in Theorem \ref{thm:dwell}. For the proof of the latter result,  it remains to show the following lemma, using the definition of~$S$ given in Section~\ref{sec:spc}.
\begin{lemma}
	Let~$v$ be a vertex of~$C_{\overline{D(s)_\SP},\Pi}$. Then there exists~$u\in \overline{D(s)_\SP}$ with~$\Pi(u)=v$ such that~$u$ switches only in~$S$.
\end{lemma}
\begin{proof}
	Choose~$c\in\R^M$ such that~$v$ is the unique minimizer of~$c^\top
	v$ subject to~$v\in C_{\overline{D(s)_\SP},\Pi}$. Moreover, choose any~$u\in \overline{D(s)_\SP}$
	with~$\Pi(u)=v$ as well as a sequence $\{u^m\}_{m\in \N}\subset D(s)_\SP$ such that $u^m \to u$ in $L^p(0,T)$. Let~$t^m_1,\dots,t^m_\sigma$ be the switching points of~$u^m$ for $m\in \N$, i.e., let $0\le t^m_1\le\dots\le t^m_\sigma<\infty$ such that $u_{t^m_1,\dots,t^m_\sigma}=u^m$. Then there exists a subsequence of~$t^m := (t_1^m, \ldots, t_\sigma^m)$ that converges to some $t\in \R^\sigma$ with $0\leq t_1 \leq \cdots \leq t_\sigma <\infty$ and, thanks to Lemma~\ref{lem:rep}, we have~$u_{t_1,\ldots,t_\sigma}=u$.
		For the following, for~$j=1,\dots,\sigma$ and $m\in
        \N\cup\{\infty\}$, we define
	$$S^{m}_j:=\{t^m_\ell\colon
        \ell\in\{1,\dots,\sigma\},\ t^m_\ell-t^m_j=s(\ell-j)\}\;,$$
        where we set~$t^\infty:=t$. The set~$S_j^m$ thus contains all
        switching points in~$t^m$ that have the minimal possible
        distance to~$t^m_j$.
	
	Assume
	first that $t_j\in(a_i,b_i)\setminus S$ for
	some~$i\in\{1,\dots,M\}$ and some ~$j\in\{1,\dots,\sigma\}$. Due to~$t^m \to t$ in $\R^\sigma$, we deduce for $m$ sufficiently large that $t_j^m\in (t_j-\tfrac{\varepsilon}{2},t_j+\tfrac{\varepsilon}{2})$, where $\varepsilon>0$ is given by $\varepsilon:=\min_{q\in S}|t_j-q|>0$. Then~$t_j^m\not\in S$ and $S^{m}_j\cap S=\emptyset$ by
	definition of~$S$. Now all points in $S^{m}_j$ can be shifted by some
	$0<\delta<\tfrac{\varepsilon}{2}$, in both directions, maintaining
	feasibility with respect to~$D(s)_\SP$, since none of these points is shifted to one of the fixing points~$\tau_1,\ldots,\tau_N$.  Consequently, all points in $S^{\infty}_j$ can be slightly shifted simultaneously in both directions, maintaining
	feasibility with respect to~$\overline{D(s)_\SP}$ and without any of these points
	leaving or entering any of the intervals~$I_1,\dots,I_M$ or
	$[0,T]$. This shifting changes the value of~$c^\top\Pi(u)$
	linearly, compare \cite[Thm.~3.12]{partI}, which is
	a contradiction to unique optimality of~$v$.
	
	We have thus shown that all
	switching points of~$u$ are either in~$S$ or outside of any interval~$I_i$. Let~$t_j\not\in
	S$ be any switching point of~$u$ not belonging to any
	interval~$I_i$. Then, for sufficiently large $m$, we have $t_j^m\notin S$ and $t^m_j\notin I_i$ for any $i\in\{1,\ldots,M\}$. The idea is, as in the proof of~\cite[Lemma 3.13]{partI}, to shift the switching point $t_j^m\notin S$ for each $m$ to the next point on the
        left belonging to~$S$, but if this point belongs to $[0,T]\cap(\Z s+(\{\tau_j\colon
	j=1,\dots,N\}))$, we can only shift~$t^m_j$ arbitrarily close to the latter point in order to maintain feasibility in $D(s)_\SP$. For small enough $\delta >0$, we thus shift 
	all switching points in~$S^{m}_j$ simultaneously
	to the left until \begin{equation}\label{eq:shift}\operatorname{dist}(S^{m}_j,S):=\min_{p\in S^{m}_j, q\in S}|p-q|=\delta,\end{equation} taking into account that the set~$S^{m}_j$ may increase when~$t^m_j$ decreases. Consequently, for all~$\delta$, we obtain another
	sequence~$\{u_{\delta}^{m}\}_{m\in \N}$. By construction, no switching point is
	moved beyond the next point in~$S$ to the left of its original
	position and no switching point is moved to any of the fixing points $\tau_1,\ldots,\tau_N$, so that we conclude $u_{\delta}^{m}(\tau_j)=c_j$ for $j=1,\ldots,N$ and thus $u_{\delta}^{m}\in D(s)_\SP$. 
	In particular, none of the switching points being moved
	enters any of the intervals~$I_i$, so that we
	derive
	\begin{equation}\label{eq:Piuem}
	\Pi(u_{\delta}^{m})=\Pi(u^m)\to \Pi(u)=v\qquad \mbox{for } m\to \infty\end{equation} by the continuity of the projection $\Pi$. We know that $\{u_{\delta}^{m}\}_{m\in \N}$ is a bounded sequence
	in~$BV(0,T)$ and hence by~\cite[Thm.~10.1.3 and Thm.~10.1.4]{ATT14}  there exists a
	strongly convergent subsequence, which we again denote by
	$\{u_{\delta}^{m}\}_{m\in \N}$, such that $ u_{\delta}^{m} \to u_{\delta}\in
	\overline{D(s)_\SP} \text{ for } m\to \infty$.
        By \eqref{eq:Piuem} and the continuity of $\Pi$, we obtain $\Pi(u_{\delta})=v$ for~$\delta>0$. 
	Now $\{u_{\delta}\colon \delta>0\}\subset\overline{D(s)_\SP}$ is bounded in~$BV(0,T)$ as well, so that it contains an accumulation point~$u'\in\overline{D(s)_\SP}$ and, again by the continuity of the projections, we have $\Pi(u')=v$. Thanks to \eqref{eq:shift}, $u'$ then has at least one switching point more in $S$ than $u$, but still satisfies $\Pi(u')=v$. By repeatedly applying the same
 	modification, we eventually obtain a function projecting to~$v$ with all
 	switching points in~$S$.
\end{proof}

\fontsize{9}{10.5}\selectfont
\bibliographystyle{siam}
\bibliography{reference}
\end{document}